\documentclass[final,leqno]{siamltex}

\usepackage[leqno]{amsmath}
\usepackage{graphicx}
\usepackage{graphics}
\usepackage{epstopdf}
\usepackage{subfigure}
\usepackage{threeparttable}

\usepackage{bm}
\usepackage{amssymb}
\usepackage{leftidx}
\usepackage{empheq}
\usepackage{color}
\allowdisplaybreaks
\renewcommand{\theequation}{\arabic{section}.\arabic{equation}}
\numberwithin{equation}{section}
\renewcommand\arraystretch{1.5}
\newtheorem{remark}{Remark}[section]
\title{Energy stable schemes for gradient flows based on novel auxiliary variable with energy bounded above.
        \thanks{
We would like to acknowledge the assistance of volunteers in putting together this example manuscript and supplement. The author thanks for the financial support from China Scholarship Council.}}

      \author{Zhengguang Liu
             \thanks{Corresponding author. School of Mathematics and Statistics, Shandong Normal University, Jinan, China. Email: liuzhgsdu@yahoo.com.}}
\begin{document}

\maketitle

\begin{abstract}
In this paper, we consider a novel auxiliary variable method to obtain energy stable schemes for gradient flows. The auxiliary variable based on energy bounded above does not limited to the hypothetical conditions adopted in previous approaches. We proved the unconditional energy stability for all the semi-discrete schemes carefully and rigorously. The novelty of the proposed schemes is that the computed values for the functional in square root are guaranteed to be positive. This method, termed novel auxiliary energy variable (NAEV) method does not consider any bounded below restrictions any longer. However, these restrictions are necessary in invariant energy quadratization (IEQ) and scalar auxiliary variable (SAV) approaches which are very popular methods recently. This property of guaranteed positivity is not available in previous approaches. A comparative study of classical SAV and NAEV approaches is considered to show the accuracy and efficiency. Finally, we present various 2D numerical simulations to demonstrate the stability and accuracy.
\end{abstract}



\pagestyle{myheadings}
\thispagestyle{plain}
\markboth{ZHENGGUANG LIU} {Deformed SAV approach for gradient flows}
  \section{Introduction}
Gradient flows are very effective models to simulate many physical phenomena and have been successfully applied to many fields such as mathematics, mechanics, materials science and other fields \cite{ambati2015review,guo2015thermodynamically,liu2019efficient,marth2016margination,miehe2010phase,shen2015efficient,wheeler1992phase,wheeler1993computation}. In general, gradient flows from the energetic variation of the energy functional $E(\phi)$ can be obtained as follows:
\begin{equation*}
\frac{\partial\phi}{\partial t}=\mathcal{G}\frac{\delta E}{\delta \phi},
\end{equation*}
where $\frac{\delta E}{\delta \phi}$ is variational derivative. $\mathcal{G}$ is a non-positive operator.

Usually, the free energy functional $E(\phi)$ contains a quadratic term and a integral term of a nonlinear functional, which can be written explicitly as follows \cite{shen2017new}
\begin{equation}\label{section1_energy}
E(\phi)=(\phi,\mathcal{L}\phi)+E_1(\phi),
\end{equation}
where $\mathcal{L}$ is a symmetric non-negative linear operator and $E_1(\phi)$ is nonlinear but with only lower-order derivatives than $\mathcal{L}$.

Energy stable is a very important property for gradient flows. This property is still essential for numerical schemes. Up to now, many scholars considered a series of  efficient and popular numerical approaches to construct energy stable schemes for gradient flows. For instance, one of popular approaches is the convex splitting method which was introduced by \cite{eyre1998unconditionally,shin2016first}. Another widely used approach is the linear stabilized scheme which can be found in \cite{shen2010numerical,yang2017numerical}. Recently, the invariant energy quadratization (IEQ) approach which were proposed by \cite{shen2018scalar,yang2016linear} has been proven to be very powerful ways to construct energy stable schemes. However, IEQ approach still has some drawbacks which can be seen in \cite{shen2017new,shen2018scalar}. Based on the core idea of the IEQ approach, Shen et. al. \cite{shen2017new,shen2018scalar} introduced the so-called scalar auxiliary variable (SAV) approach. This method can almost inherit all advantages of IEQ approach but also overcome most of its shortcomings.

In order to show and give a comparative study for our novel auxiliary energy variable approach, we provide below a brief review of SAV approach to construct energy stable schemes for gradient flows. The key of the SAV approach is to transform the nonlinear potential into a simple quadratic form. This transformation makes the nonlinear term much easier to handle. Assuming that $E_1(\phi)$ is bounded from below. It means that there exists a constant $C$ to make $E_1(\phi)+C>0$. Then, define a scalar auxiliary variable $r(t)=\sqrt{E_1(\phi)+C}$. An equivalent system of gradient flows can be rewritten as follows
\begin{equation}\label{section1_sav1}
  \left\{
   \begin{array}{rll}
\displaystyle\frac{\partial \phi}{\partial t}&=&\mathcal{G}\mu\\
\mu&=&\displaystyle\mathcal{L}\phi+\frac{r}{\sqrt{E_1(\phi)+C}}U(\phi),\\
r_t&=&\displaystyle\frac{1}{2\sqrt{E_1(\phi)+C}}\int_{\Omega}U(\phi)\phi_td\textbf{x},
   \end{array}
   \right.
\end{equation}
where $U(\phi)=\frac{\delta E_1}{\delta \phi}$.

Taking the inner products of the above equations with $\mu$, $\phi_t$ and $2r$, respectively, we obtain that the above equivalent system satisfies a modified energy dissipation law:
\begin{equation*}
\frac{d}{dt}\left[(\phi,\mathcal{L}\phi)+r^2\right]=(\mathcal{G}\mu,\mu)\leq0.
\end{equation*}

A second-order semi-discrete scheme based on the Crank-Nicolson method, reads as follows
\begin{equation}\label{section1_sav2}
  \left\{
   \begin{array}{rll}
\displaystyle\frac{\phi^{n+1}-\phi^{n}}{\Delta t}&=&\mathcal{G}\mu^{n+1/2},\\
\mu^{n+1/2}&=&\displaystyle\mathcal{L}\left(\frac{\phi^{n+1}+\phi^n}{2}\right)+\frac{r^{n+1}+r^n}{2\sqrt{E_1(\tilde{\phi}^{n+1/2})+C}}U(\tilde{\phi}^{n+1/2}),\\
\displaystyle\frac{r^{n+1}-r^n}{\Delta t}&=&\displaystyle\frac{1}{2\sqrt{E_1(\tilde{\phi}^{n+1/2})+C}}\int_{\Omega}U(\tilde{\phi}^{n+1/2})\frac{\phi^{n+1}-\phi^{n}}{\Delta t}d\textbf{x},
   \end{array}
   \right.
\end{equation}

It is not difficult to prove the above scheme is unconditionally energy stable in the sense that
\begin{equation*}
\aligned
\left[\frac12(\mathcal{L}\phi^{n+1},\phi^{n+1})+|r^{n+1}|^2\right]-\left[\frac12(\mathcal{L}\phi^{n},\phi^{n})+|r^{n}|^2\right]\leq\Delta t(\mathcal{G}\mu^{n+1/2},\mu^{n+1/2})\leq0,
\endaligned
\end{equation*}

The SAV approach is a very efficient and powerful way to construct energy stable schemes and it is easy to calculate. However, we think there is still one point to optimize. As we all know, the SAV approach needs an assumption that $E_1(\phi)$ is bounded from below to keep the square root reasonable. It means that there is a positive constant $C$ to make $E_1(\phi)+C>0$. In calculation, we notice that different values of the constant $C$ in square root may affect the effect of simulation. Relative study can be found in \cite{lin2019numerical}. In their study, they found that the error histories for $C=0.01$ and $C=500$ exhibit quite different characteristics. The error corresponding to $C$ decreases quickly but the error corresponding to $C=500$ decreases extremely slowly at this stage. In our numerical simulation in the example 6, we find that a the constant has to satisfy $C>6100$ to keep the square root reasonable. To avoid the difference like that, in this paper, we aim to find reasonable procedure to avoid using a estimated number $C$ during the calculation. We consider a novel auxiliary variable method to obtain energy stable schemes for gradient flows. The property of dissipative energy law means that the energy bounded above restriction is the initial energy $E_0$. We replace $C$ with $E_0$ in square root before calculation to construct NAEV method. The auxiliary variable based on energy bounded above does not limited to the hypothetical conditions adopted in previous approaches. We proved the unconditional energy stability for the semi-discrete schemes carefully and rigorously. The contribution of this paper is that the developed NAEV approach gets rid of the assumption that $E_1(\phi)$ is bounded from below which is essential in SAV approach. A comparative study of classical SAV and NAEV approaches is considered to show the accuracy and efficiency. Finally, we present various 2D numerical simulations to demonstrate the stability and accuracy.

The paper is organized as follows. In Sect.2, we consider two NAEV approaches for gradient flows. Then, we proved the unconditional energy stability for the semi-discrete scheme. In Sect.3, various 2D numerical simulations are demonstrated to verify the accuracy and efficiency of our proposed schemes.

\section{NAEV approaches for gradient flows}
In this section, we will consider some efficient procedures to give two NAEV approaches for gradient flows to get rid of the bounded from below restriction which is essential in SAV approach. For gradient flows, the dissipative energy law means $\frac{d}{dt}E(\phi)\leq0$. Then, an obvious property will hold as follows
\begin{equation}\label{section2_dsav1}
E(\phi(\textbf{x},0))\geq E(\phi(\textbf{x},t)), \quad \forall \textbf{x}\in\Omega,t\geq0.
\end{equation}

Next, we will take advantage of above property \eqref{section2_dsav1} to develop two deformed SAV approaches. Considering the definition of the energy in \eqref{section1_energy} and noting that $\mathcal{L}$ is a symmetric non-negative linear operator, it is not difficult to obtain the following inequality
\begin{equation}\label{section2_dsav2}
E(\phi(\textbf{x},0))-E_1(\phi(\textbf{x},t))=E(\phi(\textbf{x},0))-E(\phi(\textbf{x},t))+(\phi,\mathcal{L}\phi)\geq(\phi,\mathcal{L}\phi)\geq0, \quad \forall \textbf{x}\in\Omega,t\geq0.
\end{equation}
Then, define a novel auxiliary energy variable (NAEV)
\begin{equation}\label{section2_dsav_r1}
r(t)=\sqrt{E(\phi(\textbf{x},0))-E_1(\phi(\textbf{x},t))+\kappa},
\end{equation}
where $\kappa$ is an arbitrary sufficiently small enough non-negative constant just to ensure $r(t)\neq0$ strictly.

\begin{remark}
The initial $E(\phi(\textbf{x},0))$ can be computed by initial condition before calculation. $\kappa$ is an arbitrary constant. In calculation, we can choose $\kappa=0$ or very small constant such as $\kappa=1e-15$. This definition of the auxiliary energy variable $r(t)$ avoids using a estimated number $C$ during the calculation efficiently.
\end{remark}

For the sake of brevity, we set $E_0=E(\phi(\textbf{x},0))$. An equivalent system of gradient flows can be rewritten as follows
\begin{equation}\label{section2_dsav3}
  \left\{
   \begin{array}{rll}
\displaystyle\frac{\partial \phi}{\partial t}&=&\mathcal{G}\mu\\
\mu&=&\displaystyle\mathcal{L}\phi+\frac{r}{\sqrt{E_0-E_1(\phi)+\kappa}}U(\phi),\\
r_t&=&\displaystyle-\frac{1}{2\sqrt{E_0-E_1(\phi)+\kappa}}\int_{\Omega}U(\phi)\phi_td\textbf{x}.
   \end{array}
   \right.
\end{equation}
Taking the inner products of the above equations with $\mu$, $\phi_t$ and $2r$, respectively, we obtain that the above equivalent system satisfies a modified energy dissipation law:
\begin{equation*}
\frac{d}{dt}\left[(\phi,\mathcal{L}\phi)-|r|^2\right]=(\mathcal{G}\mu,\mu)\leq0.
\end{equation*}

Before giving a semi-discrete formulation, we let $N>0$ be a positive integer and set
\begin{equation*}
\Delta t=T/N,\quad t^n=n\Delta t,\quad \text{for}\quad n\leq N.
\end{equation*}
\subsection{The first-order NAEV scheme for gradient flows}
A first order scheme for solving the system \eqref{section2_dsav3} can be readily derived by the backward Euler¡¯s method. The first-order scheme can be written as follows:
\begin{equation}\label{section2_dsav4}
  \left\{
   \begin{array}{rll}
\displaystyle\frac{\phi^{n+1}-\phi^{n}}{\Delta t}&=&\mathcal{G}\mu^{n+1},\\
\mu^{n+1}&=&\displaystyle\mathcal{L}\phi^{n+1}+\frac{r^{n+1}+r^n}{2\sqrt{E_0-E_1(\tilde{\phi}^{n+1})+\kappa}}U(\tilde{\phi}^{n+1}),\\
\displaystyle\frac{r^{n+1}-r^n}{\Delta t}&=&\displaystyle-\frac{1}{2\sqrt{E_0-E_1(\tilde{\phi}^{n+1})+\kappa}}\int_{\Omega}U(\tilde{\phi}^{n+1})\frac{\phi^{n+1}-\phi^{n}}{\Delta t}d\textbf{x},
   \end{array}
   \right.
\end{equation}
where $\tilde{\phi}^{n+1}$ is any explicit $O(\Delta t)$ approximation for $\phi(t^{n+1})$, which can be flexible according to the problem. Here, we choose $\tilde{\phi}^{n+1}=2\phi^n-\phi^{n-1}$ for $n\geq1$.
\begin{theorem}\label{section2_th_dsav}
The first-order NAEV scheme \eqref{section2_dsav4} admits a unique solution satisfying the following discrete energy dissipation law,
\begin{equation}\label{section2_dsav5}
\aligned
E_{1st}^{n+1}-E^{n}_{1st}\leq\Delta t(\mathcal{G}\mu^{n+1},\mu^{n+1})\leq0,
\endaligned
\end{equation}
where the modified discrete version of the energy \eqref{section1_energy} is defined by
\begin{equation*}
\aligned
E_{1st}^{n}=\frac12\|\phi^{n}\|_{\mathcal{L}}^2-|r^{n}|^2.
\endaligned
\end{equation*}
\end{theorem}
\begin{proof}
By taking the inner products with $\Delta t\mu^{n+1}$, $(\phi^{n+1}-\phi^n)$, and $-\Delta t(r^{n+1}+r^{n})$ for the three equations in the first-order scheme \eqref{section2_dsav4} respectively and some simple calculations, we obtain
\begin{equation}\label{section2_dsav6}
\aligned
(\phi^{n+1}-\phi^{n},\mu^{n+1})=\Delta t(\mathcal{G}\mu^{n+1},\mu^{n+1}),
\endaligned
\end{equation}
\begin{equation}\label{section2_dsav7}
\aligned
(\phi^{n+1}-\phi^{n},\mu^{n+1})=
&\displaystyle(\mathcal{L}\phi^{n+1},\phi^{n+1}-\phi^{n})+\frac12\left(\frac{r^{n+1}+r^n}{\sqrt{E_0-E_1(\tilde{\phi}^{n+1})+\kappa}}U(\tilde{\phi}^{n+1}),\phi^{n+1}-\phi^{n}\right),
\endaligned
\end{equation}
and
\begin{equation}\label{section2_dsav8}
\aligned
|r^{n}|^2-|r^{n+1}|^2=\displaystyle\frac12\left(\frac{r^{n+1}+r^n}{\sqrt{E_0-E_1(\tilde{\phi}^{n+1})+\kappa}}U(\tilde{\phi}^{n+1}),\phi^{n+1}-\phi^{n}\right).
\endaligned
\end{equation}

Applying the following identity
\begin{equation*}
\aligned
(\mathcal{L}x,x-y)=\frac12\|x\|_{\mathcal{L}}^2-\frac12\|y\|_{\mathcal{L}}^2+\frac12\|x-y\|_{\mathcal{L}}^2,
\endaligned
\end{equation*}
where we define $\|x\|_{\mathcal{L}}^2=(\mathcal{L}x,x)$. Let $x=\phi^{n+1}$, $y=\phi^{n}$, it is easy to obtain
\begin{equation}\label{section2_dsav9}
\aligned
(\mathcal{L}\phi^{n+1},\phi^{n+1}-\phi^{n})=\frac12\|\phi^{n+1}\|_{\mathcal{L}}^2-\frac12\|\phi^{n}\|_{\mathcal{L}}^2+\frac12\|\phi^{n+1}-\phi^{n}\|_{\mathcal{L}}^2
\endaligned
\end{equation}
Combining the equations \eqref{section2_dsav7}-\eqref{section2_dsav9} with \eqref{section2_dsav6} and noting that $\mathcal{G}$ is a non-positive operator, we obtain that
\begin{equation*}
\aligned
0\geq\Delta t(\mathcal{G}\mu^{n+1},\mu^{n+1})\geq\frac12\|\phi^{n+1}\|_{\mathcal{L}}^2-|r^{n+1}|^2-\frac12\|\phi^{n}\|_{\mathcal{L}}^2+|r^{n}|^2=E_{1st}^{n+1}-E^{n}_{1st}.
\endaligned
\end{equation*}
which completes the proof.
\end{proof}
\subsection{The second-order NAEV scheme for gradient flows}
In this subsection, we aim to propose a second-order energy stable semi-discrete NAEV scheme for gradient flows. For any $t^{n+1/2}\geq0$, the following inequality is also satisfied:
\begin{equation*}
\aligned
E(\phi(\textbf{x},0))-E_{1}(\phi(\textbf{x},t^{n+1/2}))
&=E(\phi(\textbf{x},0))-E(\phi(\textbf{x},t^{n+1/2}))+(\phi^{n+1/2},\mathcal{L}\phi^{n+1/2})\\
&\geq(\phi^{n+1/2},\mathcal{L}\phi^{n+1/2})\geq0, \quad \forall \textbf{x}\in\Omega,t^n\geq0.
\endaligned
\end{equation*}

A second-order semi-discrete NAEV scheme based on the Crank-Nicolson method, reads as follows
\begin{equation}\label{section2_dsav10}
  \left\{
   \begin{array}{rll}
\displaystyle\frac{\phi^{n+1}-\phi^{n}}{\Delta t}&=&\mathcal{G}\mu^{n+1/2},\\
\mu^{n+1/2}&=&\displaystyle\mathcal{L}\left(\frac{\phi^{n+1}+\phi^n}{2}\right)+\frac{r^{n+1}+r^n}{2\sqrt{E_0-E_1(\tilde{\phi}^{n+1/2})+\kappa}}U(\tilde{\phi}^{n+1/2}),\\
\displaystyle\frac{r^{n+1}-r^n}{\Delta t}&=&\displaystyle-\frac{1}{2\sqrt{E_0-E_1(\tilde{\phi}^{n+1/2})+\kappa}}\int_{\Omega}U(\tilde{\phi}^{n+1/2})\frac{\phi^{n+1}-\phi^{n}}{\Delta t}d\textbf{x},
   \end{array}
   \right.
\end{equation}
where $\tilde{\phi}^{n+\frac{1}{2}}$ is any explicit $O(\Delta t^2)$ approximation for $\phi(t^{n+\frac{1}{2}})$, which can be flexible according to the problem. Here, we choose $\tilde{\phi}^{n+\frac{1}{2}}=(3\phi^n-\phi^{n-1})/2$ for $n>0$. For $n=0$, we compute $\tilde{\phi}^{\frac{1}{2}}$ by using the following simple scheme:
\begin{equation*}
\frac{\tilde{\phi}^{\frac{1}{2}}-\phi^0}{(\Delta t)/2}=\mathcal{G}\left(\mathcal{L}\tilde{\phi}^{\frac{1}{2}}+F'(\phi^0)\right).
\end{equation*}

\begin{theorem}\label{section2_th_dsav2}
The NAEV-CN scheme \eqref{section2_dsav10} for the gradient flows is unconditionally energy stable in the sense that
\begin{equation}\label{section2_dsav11}
\aligned
E_{NAEV-CN}^{n+1}-E^{n}_{NAEV-CN}\leq\Delta t(\mathcal{G}\mu^{n+1/2},\mu^{n+1/2})\leq0,
\endaligned
\end{equation}
where the modified discrete version of the energy \eqref{section1_energy} is defined by
\begin{equation*}
\aligned
E_{NAEV-CN}^{n}=\frac12(\mathcal{L}\phi^{n},\phi^{n})-|r^{n}|^2.
\endaligned
\end{equation*}
\end{theorem}
By taking the inner products with $\Delta t\mu^{n+1/2}$, $(\phi^{n+1}-\phi^n)$, and $\Delta t(r^{n+1}+r^n)$ for the three equations in scheme \eqref{section2_dsav10} respectively and some simple calculations, it is not difficult to obtain the proof.

\subsection{The NAEV scheme with imaginary number for gradient flows}
From above two subsections, we construct two NAEV schemes for gradient flows. We find that a minus in expression of $r_t$ leads to the difference between SAV scheme and NAEV scheme. In this subsection, we try to develop a NAEV with imaginary number scheme which is exactly same with the classical SAV scheme in form. We redefine an auxiliary energy variable with imaginary number as follows:
\begin{equation}\label{section2_dsav_r2}
r(t)=i\sqrt{E(\phi(\textbf{x},0))-E_1(\phi(\textbf{x},t))+\kappa},
\end{equation}
where $i^2=-1$. It is not difficult to obtain that
\begin{equation*}
r_t=\frac{-iU(\phi)}{\sqrt{E(\phi(\textbf{x},0))-E_1(\phi(\textbf{x},t))+\kappa}}=\frac{U(\phi)}{i\sqrt{E(\phi(\textbf{x},0))-E_1(\phi(\textbf{x},t))+\kappa}}.
\end{equation*}
The advantage of above definition is that the discrete energy is exactly the same with the classical SAV scheme in form.

A novel second-order semi-discrete NAEV scheme based on the Crank-Nicolson method, reads as follows
\begin{equation}\label{section2_dsav12}
  \left\{
   \begin{array}{rll}
\displaystyle\frac{\phi^{n+1}-\phi^{n}}{\Delta t}&=&\mathcal{G}\mu^{n+1/2},\\
\mu^{n+1/2}&=&\displaystyle\mathcal{L}\left(\frac{\phi^{n+1}+\phi^n}{2}\right)+\frac{r^{n+1}+r^n}{2i\sqrt{E_0-E_1(\tilde{\phi}^{n+1/2})+\kappa}}U(\tilde{\phi}^{n+1/2}),\\
\displaystyle\frac{r^{n+1}-r^n}{\Delta t}&=&\displaystyle\frac{1}{2i\sqrt{E_0-E_1(\tilde{\phi}^{n+1/2})+\kappa}}\int_{\Omega}U(\tilde{\phi}^{n+1/2})\frac{\phi^{n+1}-\phi^{n}}{\Delta t}d\textbf{x},
   \end{array}
   \right.
\end{equation}
In calculation, it is not difficult to find that $i$ will not be used actually.
\begin{theorem}\label{section2_th_dsav3}
The NAEV-CN scheme with imaginary number \eqref{section2_dsav12} for the gradient flows is unconditionally energy stable in the sense that
\begin{equation}\label{section2_dsav13}
\aligned
E_{NAEV-CN}^{n+1}-E^{n}_{NAEV-CN}\leq\Delta t(\mathcal{G}\mu^{n+1/2},\mu^{n+1/2})\leq0,
\endaligned
\end{equation}
where the modified discrete version of the energy \eqref{section1_energy} is defined by
\begin{equation*}
\aligned
E_{NAEV-CN}^{n}=\frac12(\mathcal{L}\phi^{n},\phi^{n})+|r^{n}|^2.
\endaligned
\end{equation*}
\end{theorem}
\section{Numerical experiments}
In this section, we present some numerical examples for some classical gradient flows such as Allen-Cahn model, Cahn-Hilliard model, phase field crystal model and Swift-Hohenberg model in 2D to test our theoretical analysis which contain energy stability and convergence rates of the proposed numerical schemes. We use the finite difference method for spatial discretization for all numerical examples. In both first-order and second-order NAEV scheme, we set $\kappa=0$.

\subsection{Allen-Cahn and Cahn-Hilliard models}
In this subsection, we give some numerical examples of Allen-Cahn and Cahn-Hilliard models to validate the efficiency and accuracy of the proposed schemes. The Allen-Cahn and Cahn-Hilliard models are two of classical gradient flows and has been widely studied by many scholars \cite{chen2018accurate,chen2018power,du2018stabilized}.

Consider the following Lyapunov energy functional:
\begin{equation}\label{section5_energy1}
E(\phi)=\int_{\Omega}(\frac{\epsilon^2}{2}|\nabla \phi|^2+F(\phi))d\textbf{x},
\end{equation}
where $F(\phi)$ is double-well type potential which is defined as $F(\phi)=\frac{1}{4}(\phi^2-1)^2$.

By applying the variational approach for the free energy \eqref{section5_energy1} leads to
\begin{equation}\label{section5_e_model}
  \left\{
   \begin{array}{rlr}
\displaystyle\frac{\partial \phi}{\partial t}&=\mathcal{G}\mu,     &(\textbf{x},t)\in\Omega\times J,\\
                                          \mu&=-\Delta \phi+f(\phi),&(\textbf{x},t)\in\Omega\times J,
   \end{array}
   \right.
  \end{equation}
where $J=(0,T]$, $\mu$ is the chemical potential, the parameter $\epsilon$ represents the interface width and $f(\phi)=F^{\prime}(\phi)$. $\mathcal{G}=-1$ for the Allen-Cahn model and $\mathcal{G}=\Delta$ for the Cahn-Hilliard model.

We first perform the following numerical example to test the convergence rates of the first-order and second-order NAEV scheme.

\textbf{Example 1}: Consider the Allen-Cahn and Cahn-Hilliard equations in $\Omega=(0,1)$ with $\epsilon=0.2$ and the following initial condition \cite{ainsworth2017analysis}:
\begin{equation*}
\aligned
\phi(x,y,0)=\cos(2\pi x)\cos(2\pi y).
\endaligned
\end{equation*}

In order to test the temporal numerical accuracy, we choose $h=0.01$ and different time step size $\Delta t$. Since the exact solution is unknown, the estimated convergence rate can be calculated as follows:
\begin{equation*}
\aligned
Rate=\log_2\frac{Err_{\Delta t}^{(\Delta t)/2}}{Err_{(\Delta t)/2}^{(\Delta t)/4}},
\endaligned
\end{equation*}
where $Err_{\Delta t}^{(\Delta t)/2}=\|\Phi^{\Delta t}-\Phi^{(\Delta t)/2}\|_2$ and $\Phi^{\Delta t}$ is the numerical result vector at $T$ with the time step $\Delta t$.
\begin{table}[h!b!p!]
\small
\renewcommand{\arraystretch}{1.1}
\centering
\caption{\small The $L_2$ errors, convergence rates of Allen-Cahn and Cahn-Hilliard equations for the first-order NAEV scheme in time at $T=8e-3$.}\label{tab:tab1}
\begin{tabular}{cccccccccccc}
\hline
        &      &        &Allen-Cahn-Equ&        &&       &Cahn-Hilliard-Equ&\\
\cline{2-5}\cline{7-10}
$\Delta t$   &$4e-4$  &$2e-4$  &$1e-4$    &$5e-5$   &&$4e-4$  &$2e-4$  &$1e-4$    &$5e-5$\\
\hline
Err          &4.90e-6 &2.45e-6   &1.22e-6  &6.12e-7  &&3.24e-3 &1.68e-3 &8.58e-4   &4.32e-4\\
Rate         &-       &1.00      &1.00     &1.00     &&-       &0.95    &0.97      &0.99\\
\hline
\end{tabular}
\end{table}
\begin{table}[h!b!p!]
\small
\renewcommand{\arraystretch}{1.1}
\centering
\caption{\small The $L_2$ errors, convergence rates of Allen-Cahn and Cahn-Hilliard equations for the second-order NAEV-CN scheme in time at $T=8e-3$.}\label{tab:tab2}
\begin{tabular}{cccccccccccc}
\hline
        &      &        &Allen-Cahn-Equ&        &&       &Cahn-Hilliard-Equ&\\
\cline{2-5}\cline{7-10}
$\Delta t$   &$4e-4$  &$2e-4$    &$1e-4$   &$5e-5$   &&$4e-4$  &$2e-4$  &$1e-4$    &$5e-5$\\
\hline
Err          &2.47e-9 &6.20e-10   &1.55e-10  &3.89e-11 &&1.18e-4 &2.65e-5 &6.06e-6   &1.42e-6\\
Rate         &-       &1.99      &2.00     &1.99     &&-       &2.15    &2.12      &2.09\\
\hline
\end{tabular}
\end{table}

We consider the NAEV-CN scheme to study the phase separation behavior by using the following example.

\textbf{Example 2}: In the following, we take $\Omega=[0,1]\times[0,1]$, $\epsilon=0.01$. The initial condition is chosen as
\begin{equation*}
\aligned
\phi_0(x,y,0)=1-\tanh\frac{\sqrt{(x-0.3)^2+(y-0.5)^2}-R_0}{\sqrt{2}\epsilon}-\tanh\frac{\sqrt{(x-0.7)^2+(y-0.5)^2}-R_0}{\sqrt{2}\epsilon}.
\endaligned
\end{equation*}
with the radius $R_0=0.19$. Initially, two bubbles, centered at $(0.3,0.5)$ and $(0.7,0.5)$, respectively, are osculating.

\begin{figure}[htp]
\centering
\subfigure[t=0]{
\includegraphics[width=3.5cm,height=3.5cm]{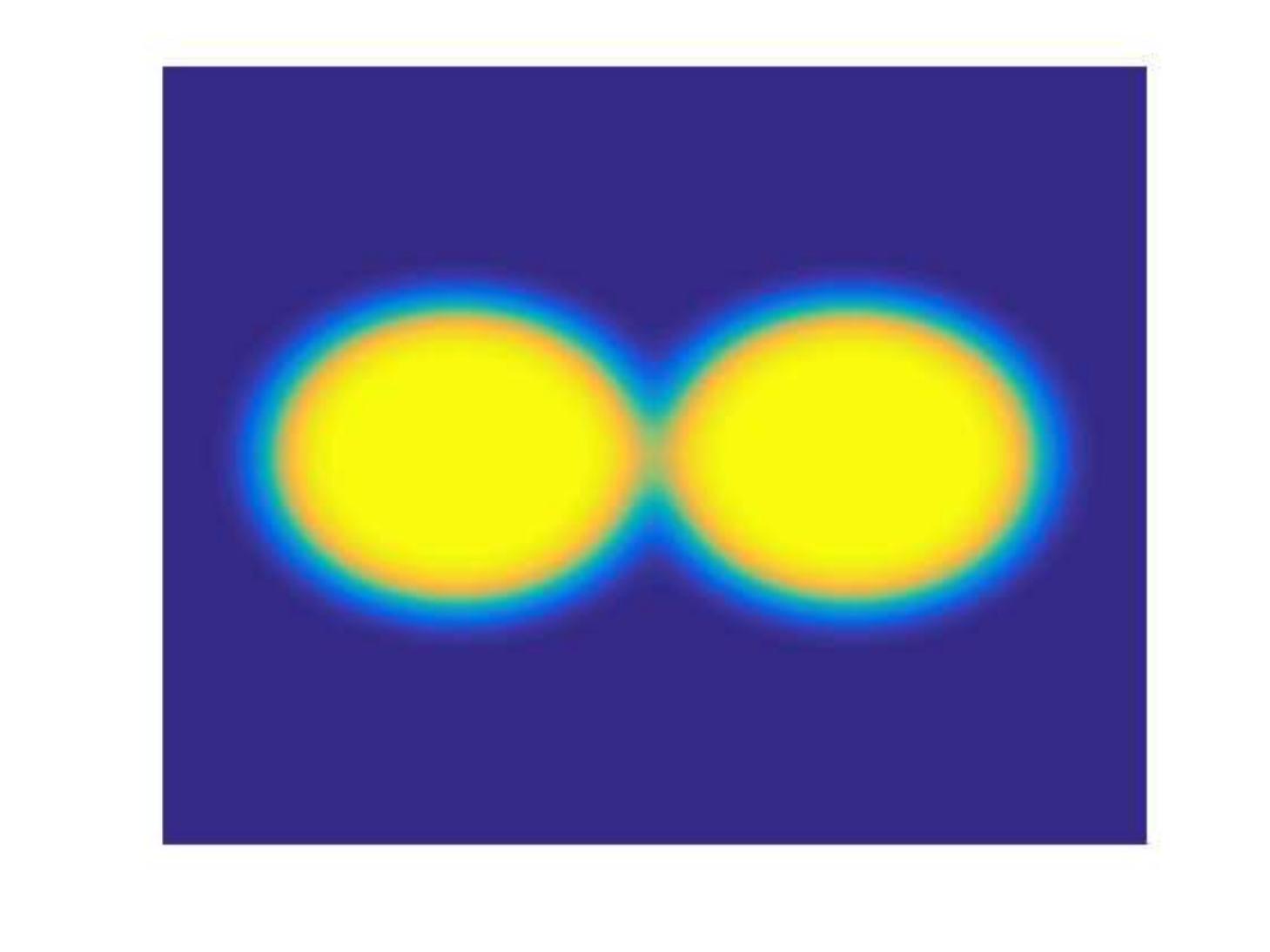}
}
\subfigure[t=1]
{
\includegraphics[width=3.5cm,height=3.5cm]{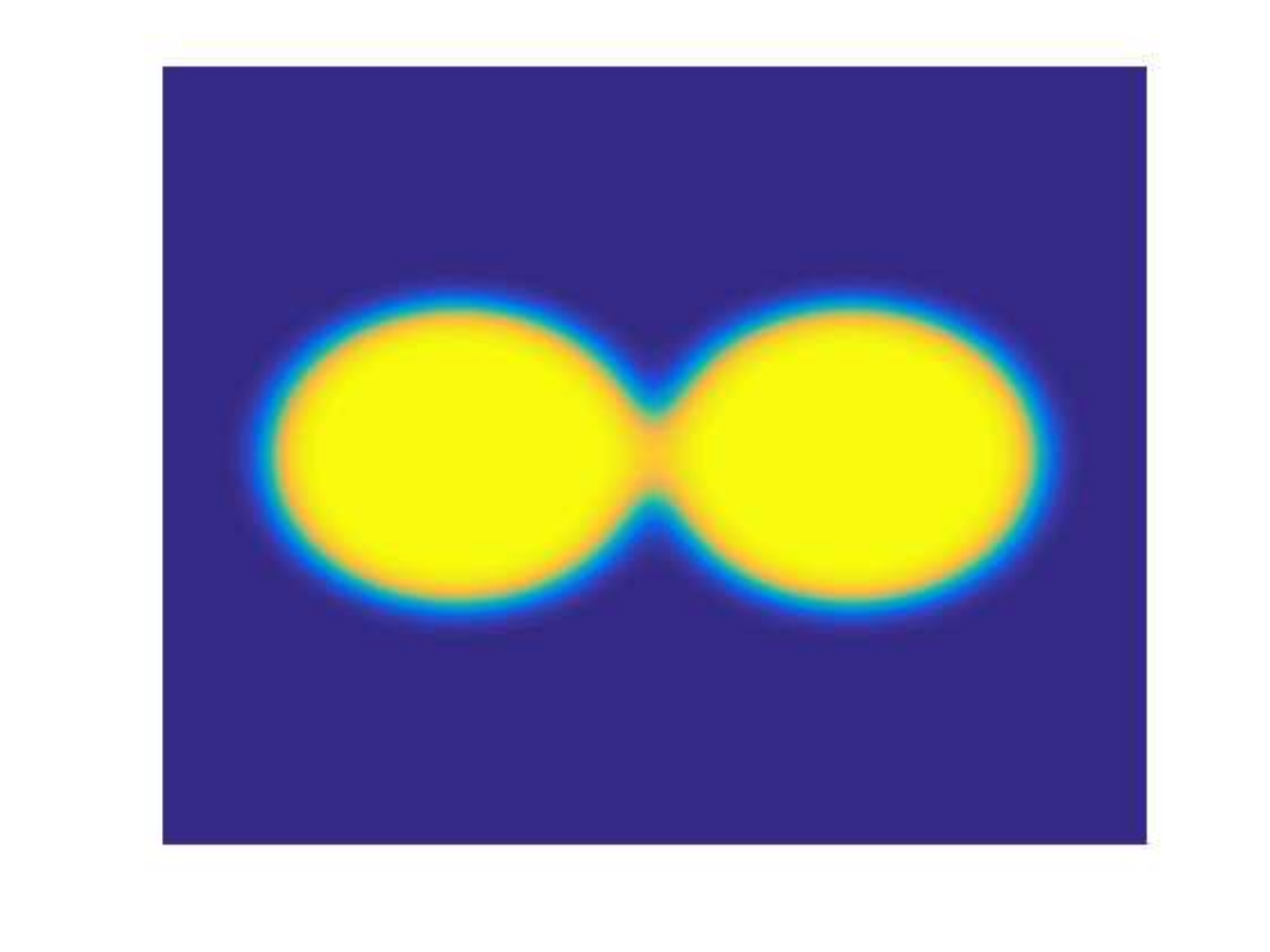}
}
\subfigure[t=3]
{
\includegraphics[width=3.5cm,height=3.5cm]{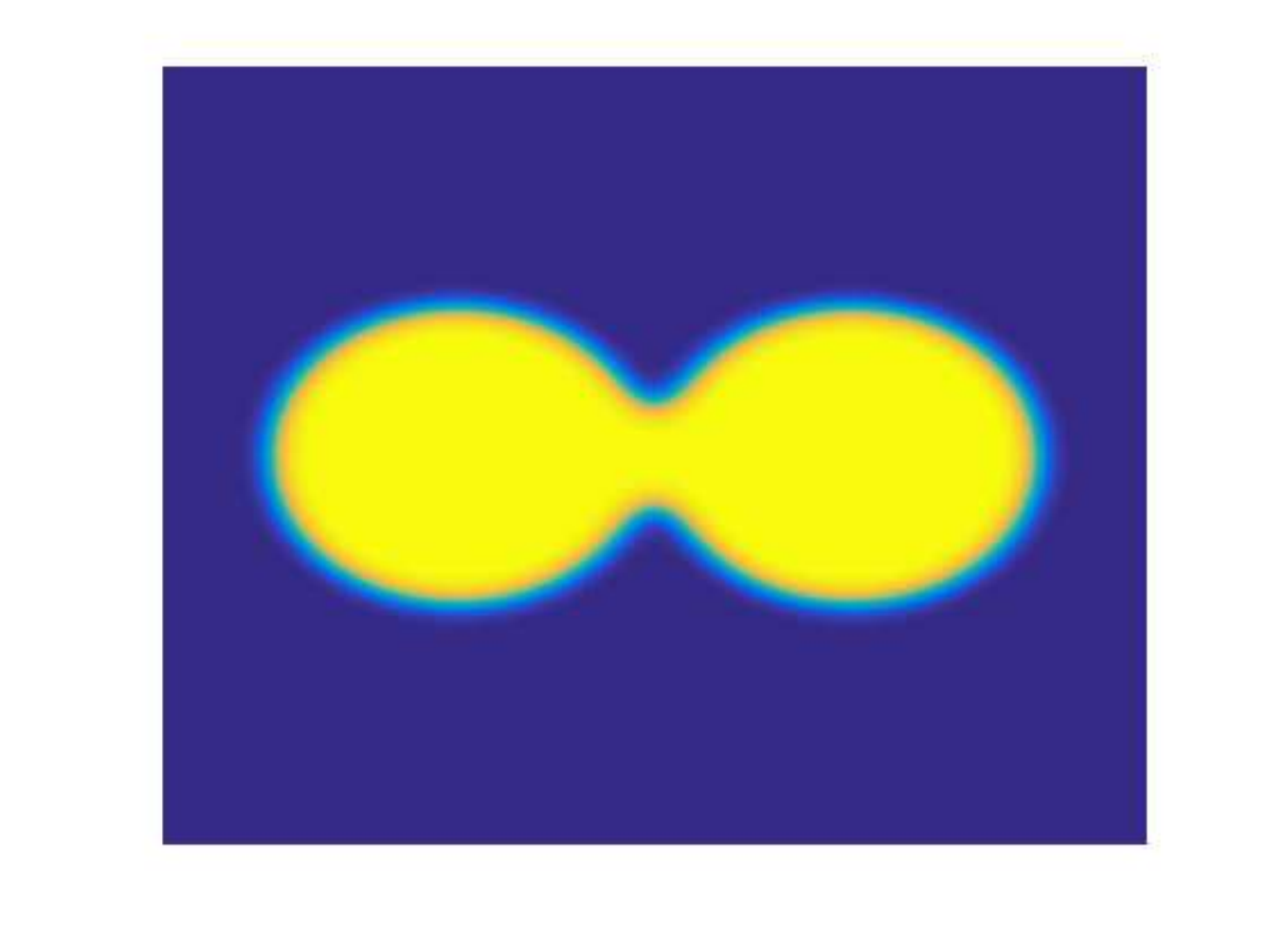}
}
\subfigure[t=14]
{
\includegraphics[width=3.5cm,height=3.5cm]{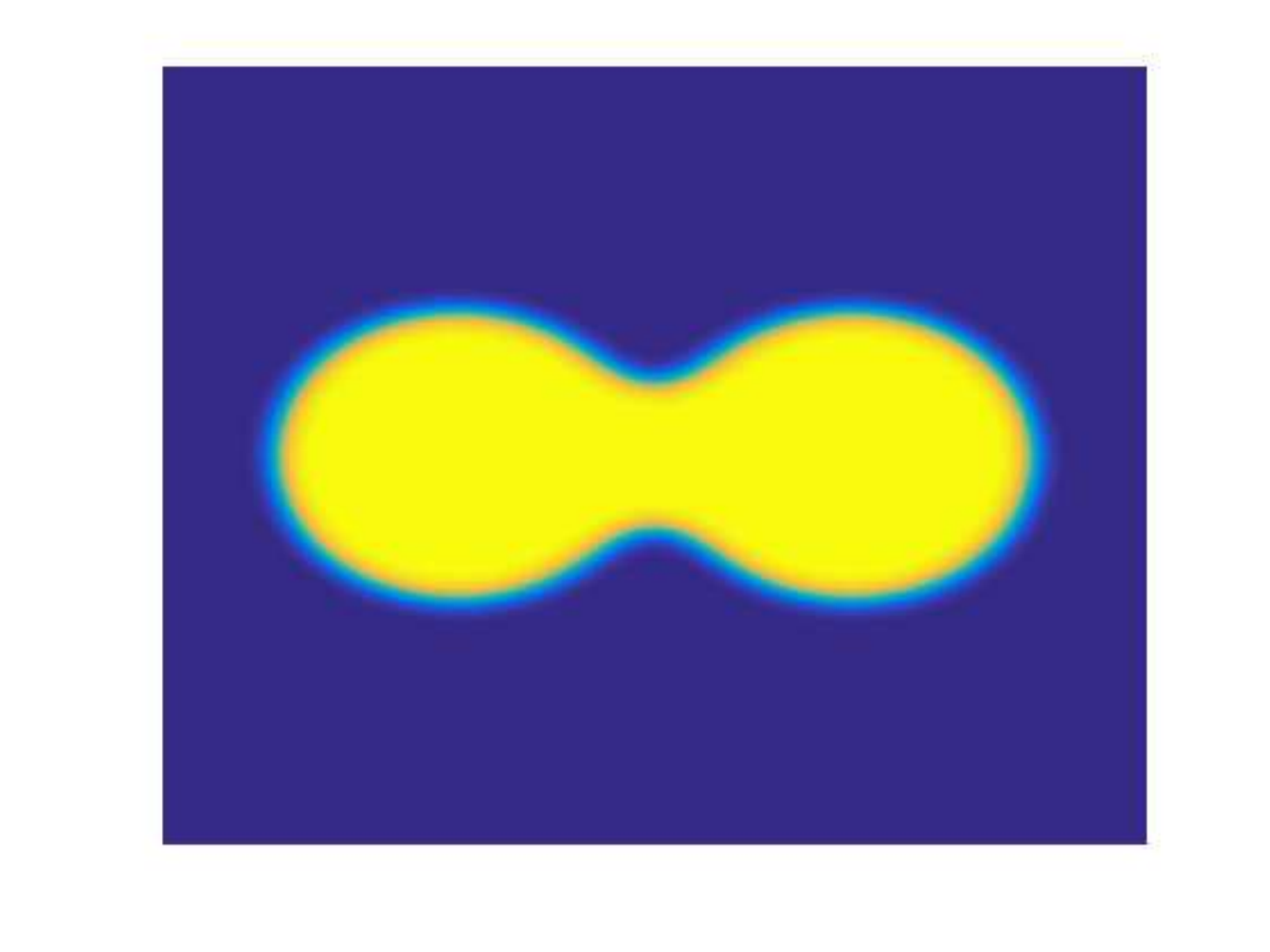}
}
\quad
\subfigure[t=60]{
\includegraphics[width=3.5cm,height=3.5cm]{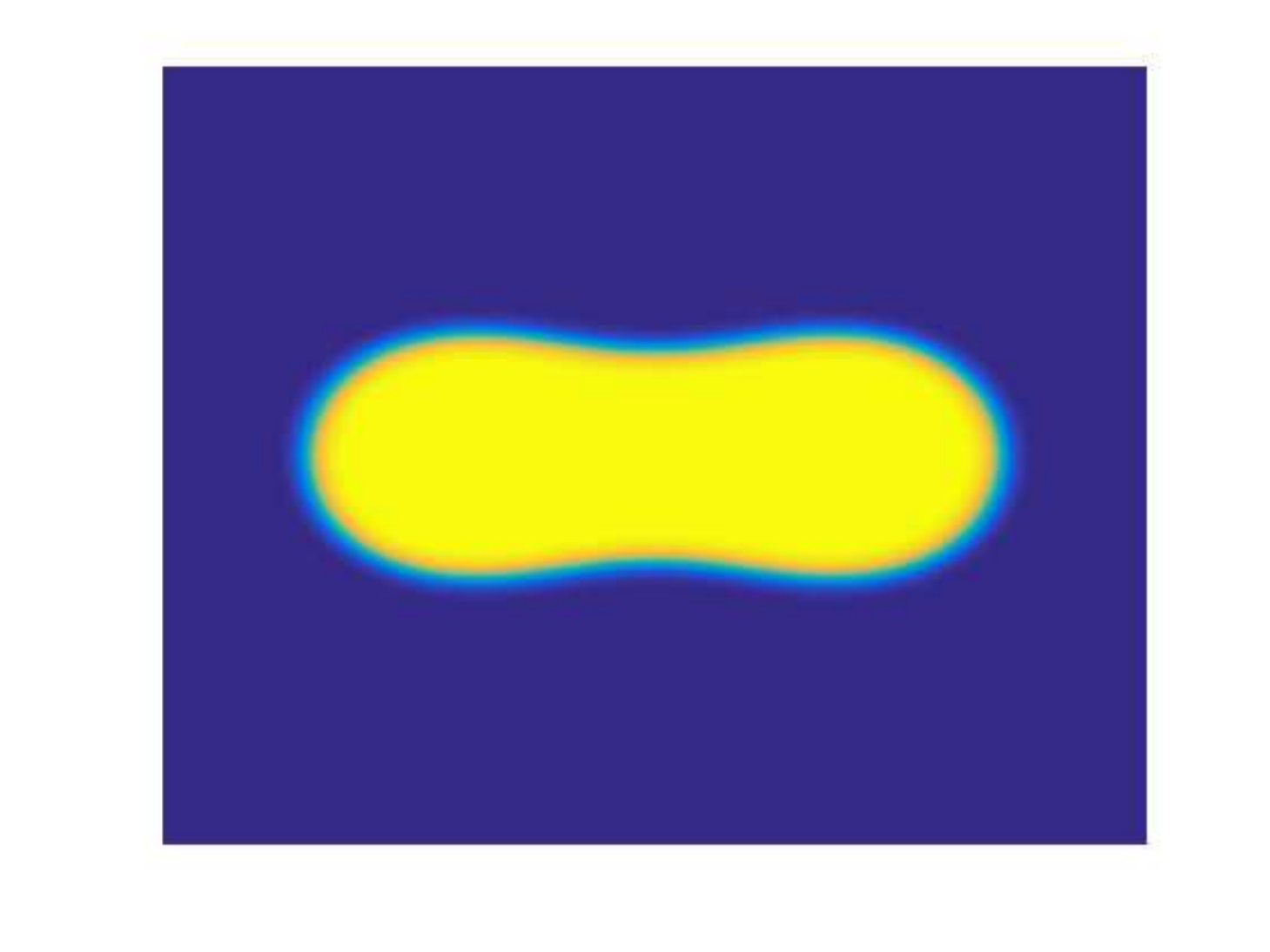}
}
\subfigure[t=130]
{
\includegraphics[width=3.5cm,height=3.5cm]{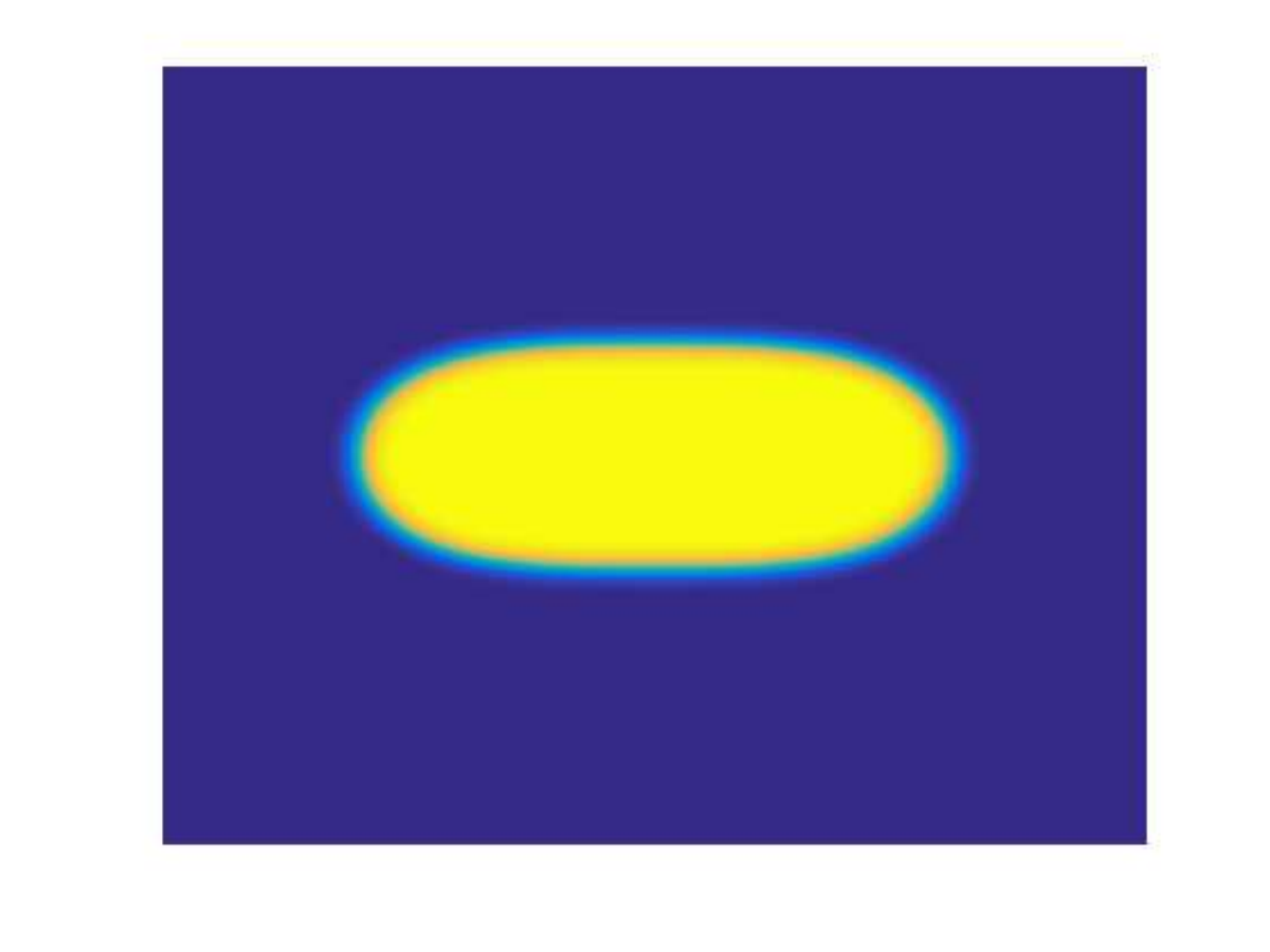}
}
\subfigure[t=240]
{
\includegraphics[width=3.5cm,height=3.5cm]{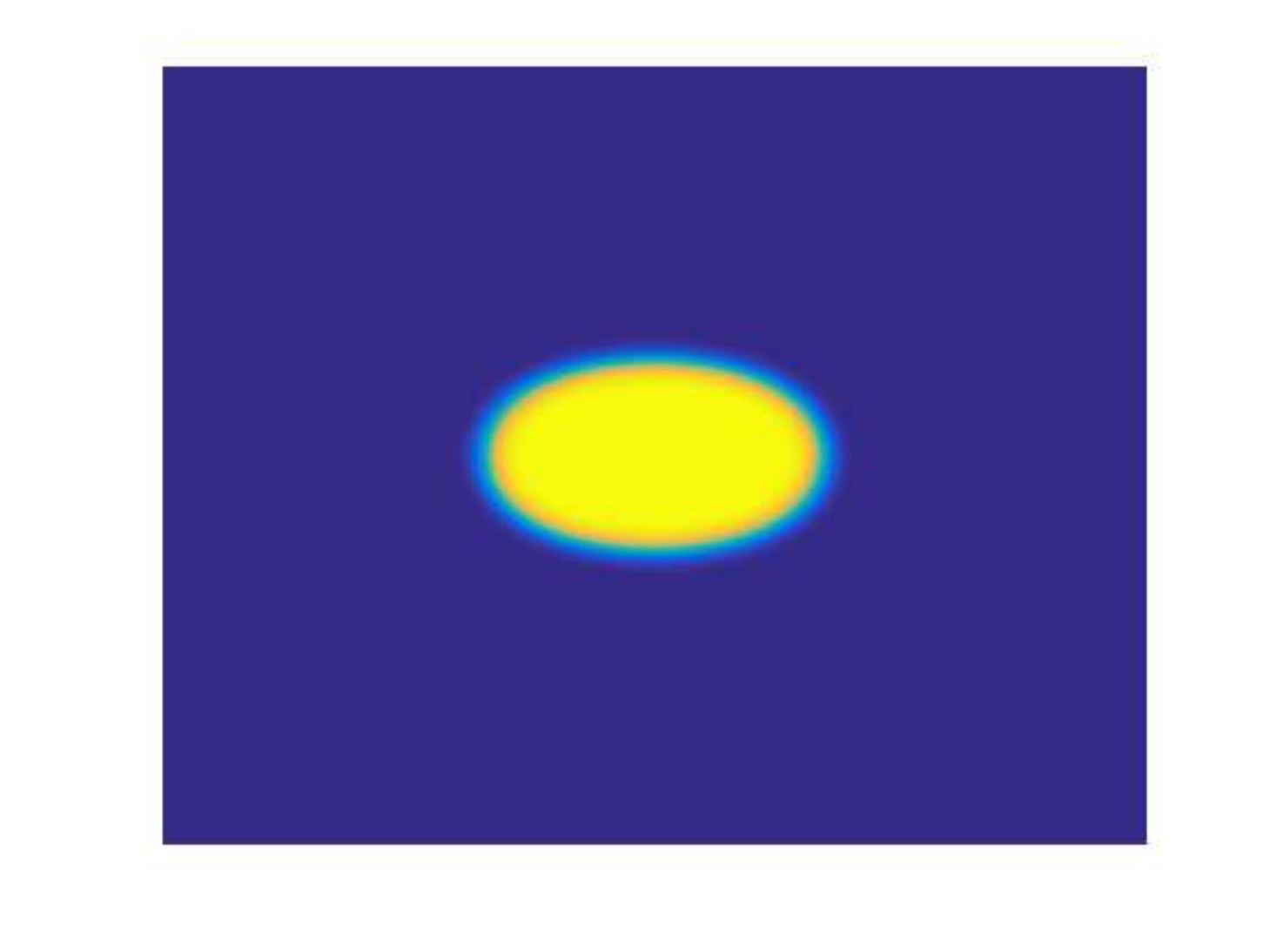}
}
\subfigure[t=340]
{
\includegraphics[width=3.5cm,height=3.5cm]{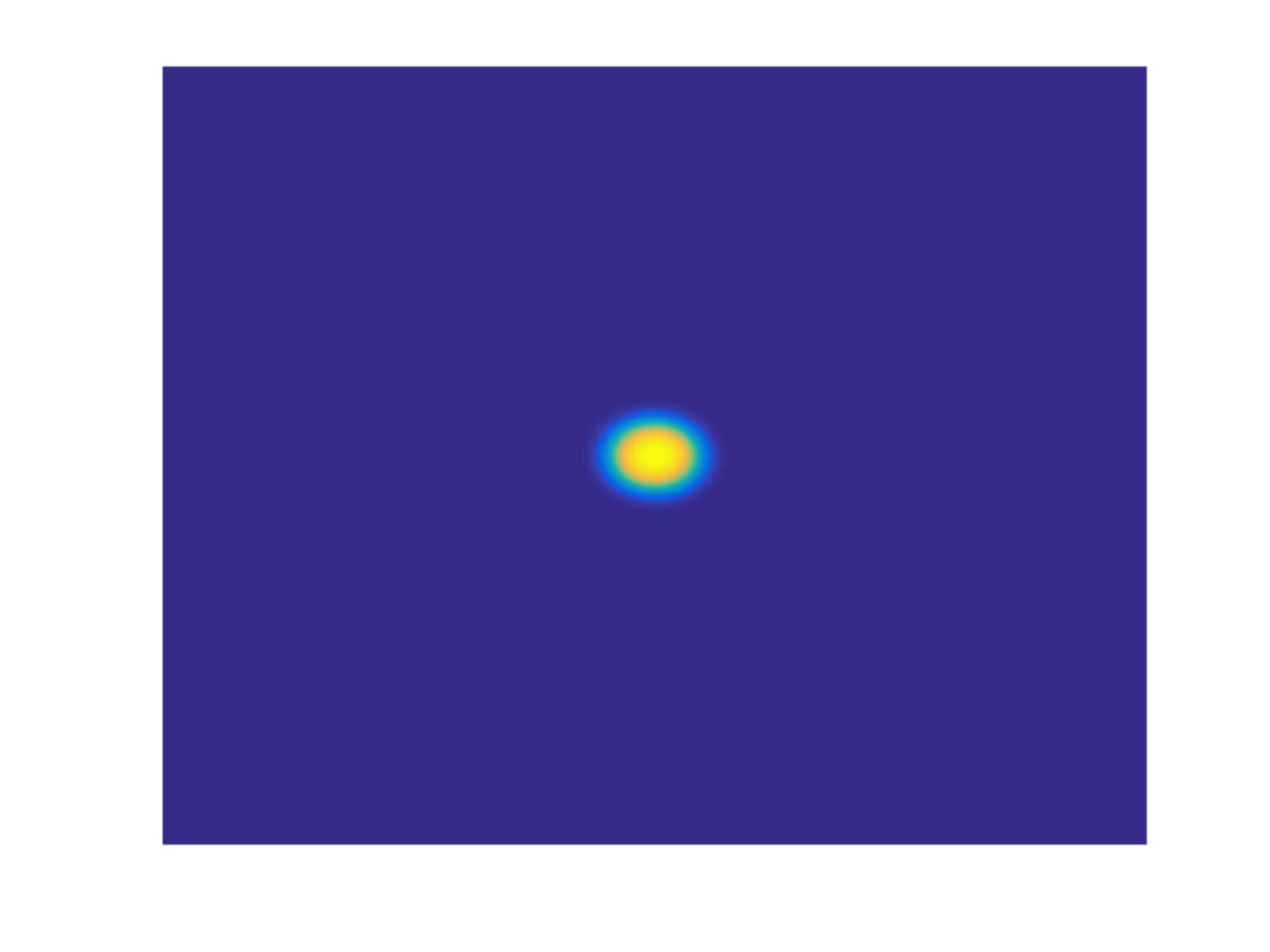}
}
\caption{Snapshots of the phase variable $\phi$ for Allen-Cahn Model are taken at t=0, 1, 3, 14, 60, 130, 240, 340. The time step is $\Delta t=0.1$.}\label{ac-fig1}
\end{figure}

\begin{figure}[htp]
\centering
\subfigure[t=0]{
\includegraphics[width=3.5cm,height=3.5cm]{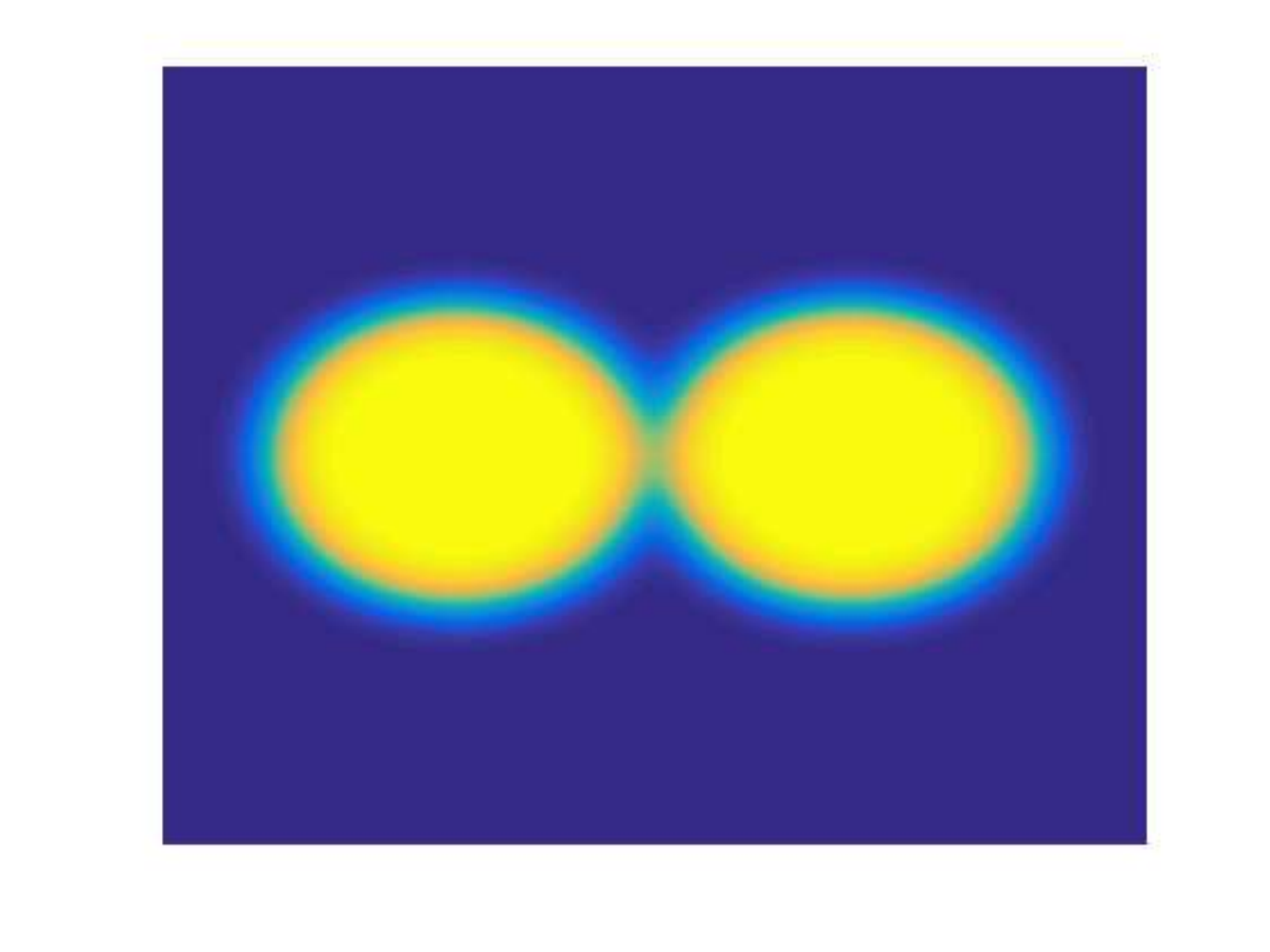}
}
\subfigure[t=8e-4]
{
\includegraphics[width=3.5cm,height=3.5cm]{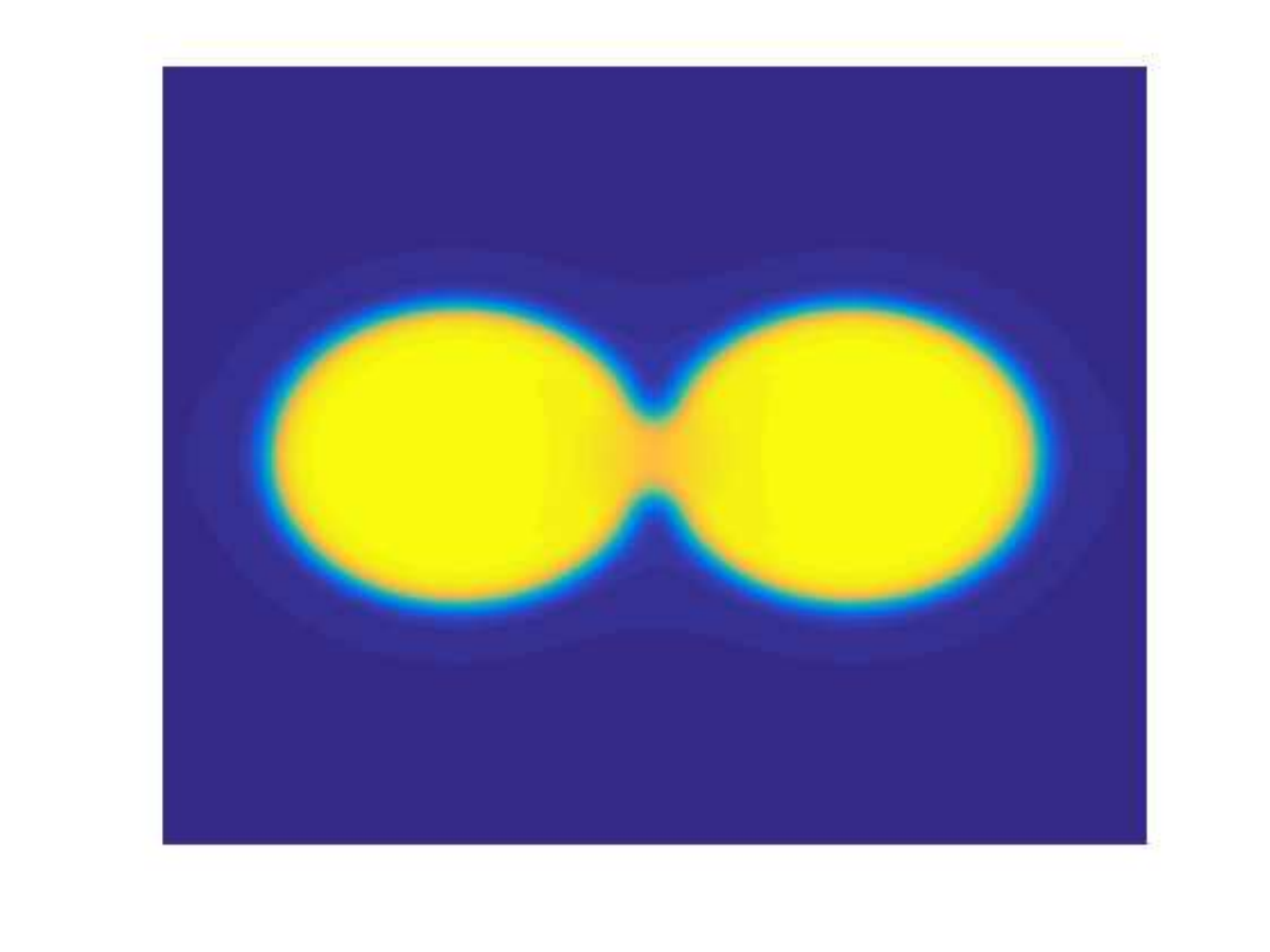}
}
\subfigure[t=4e-3]
{
\includegraphics[width=3.5cm,height=3.5cm]{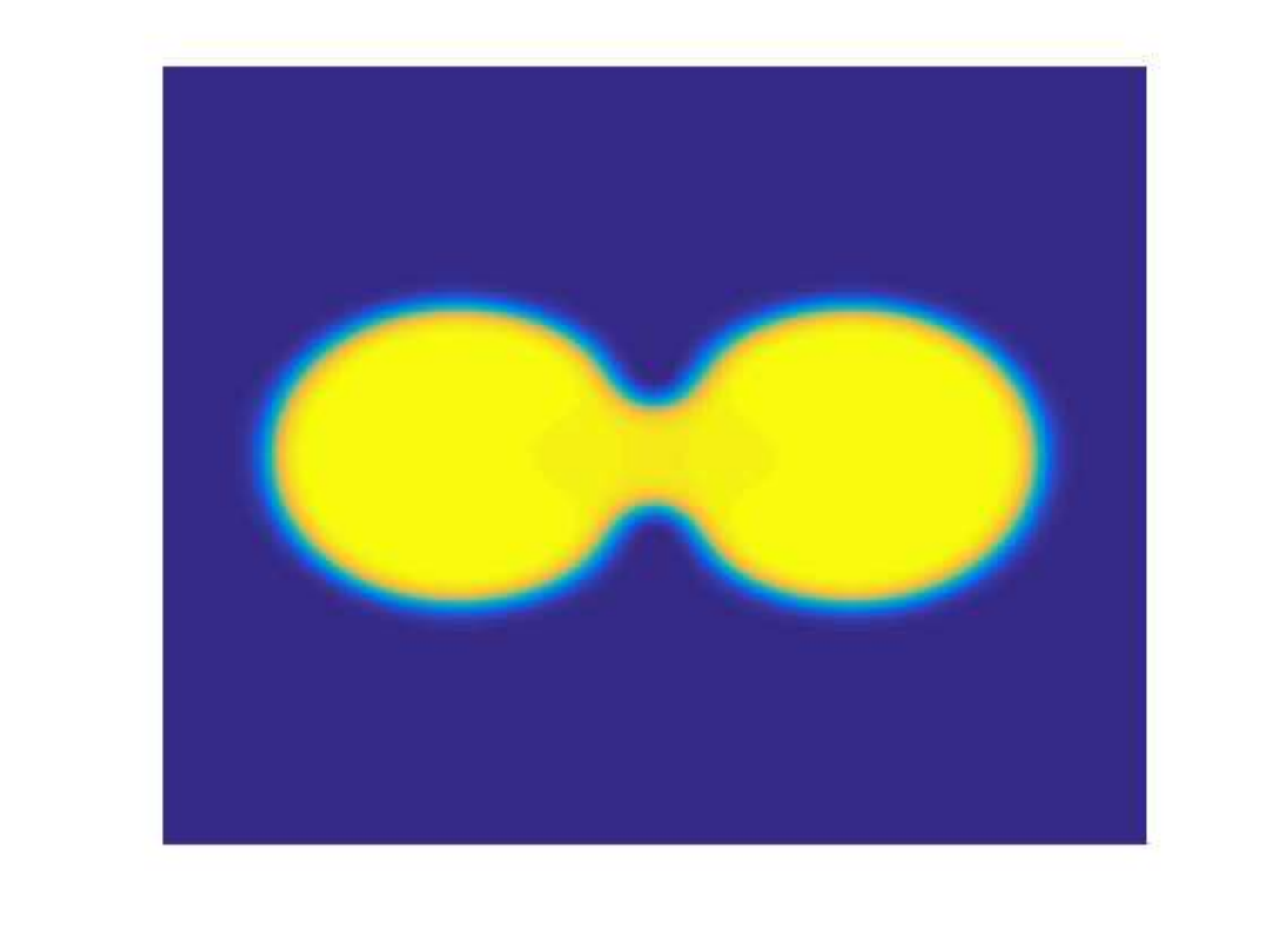}
}
\subfigure[t=0.02]
{
\includegraphics[width=3.5cm,height=3.5cm]{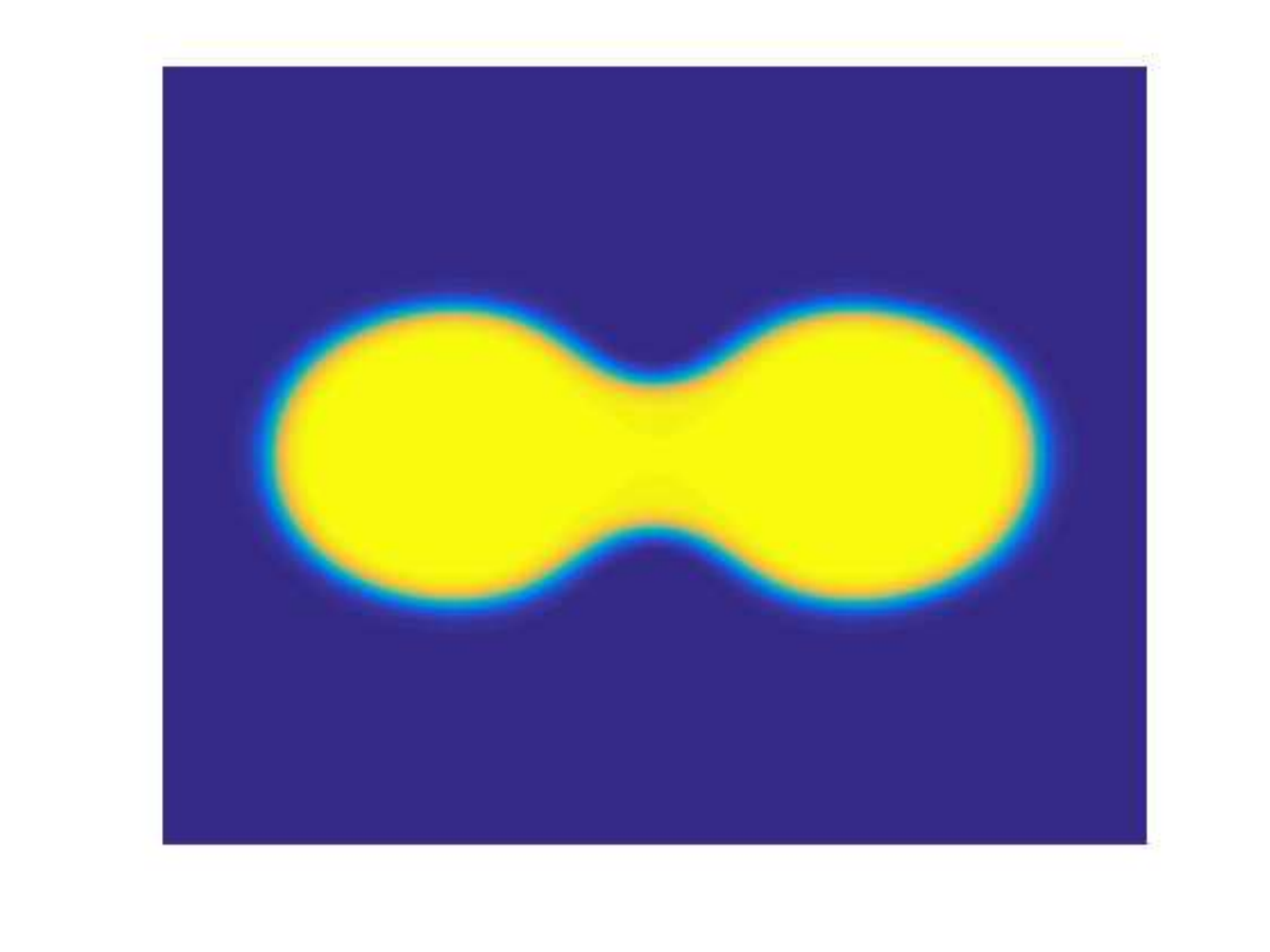}
}
\quad
\subfigure[t=0.16]{
\includegraphics[width=3.5cm,height=3.5cm]{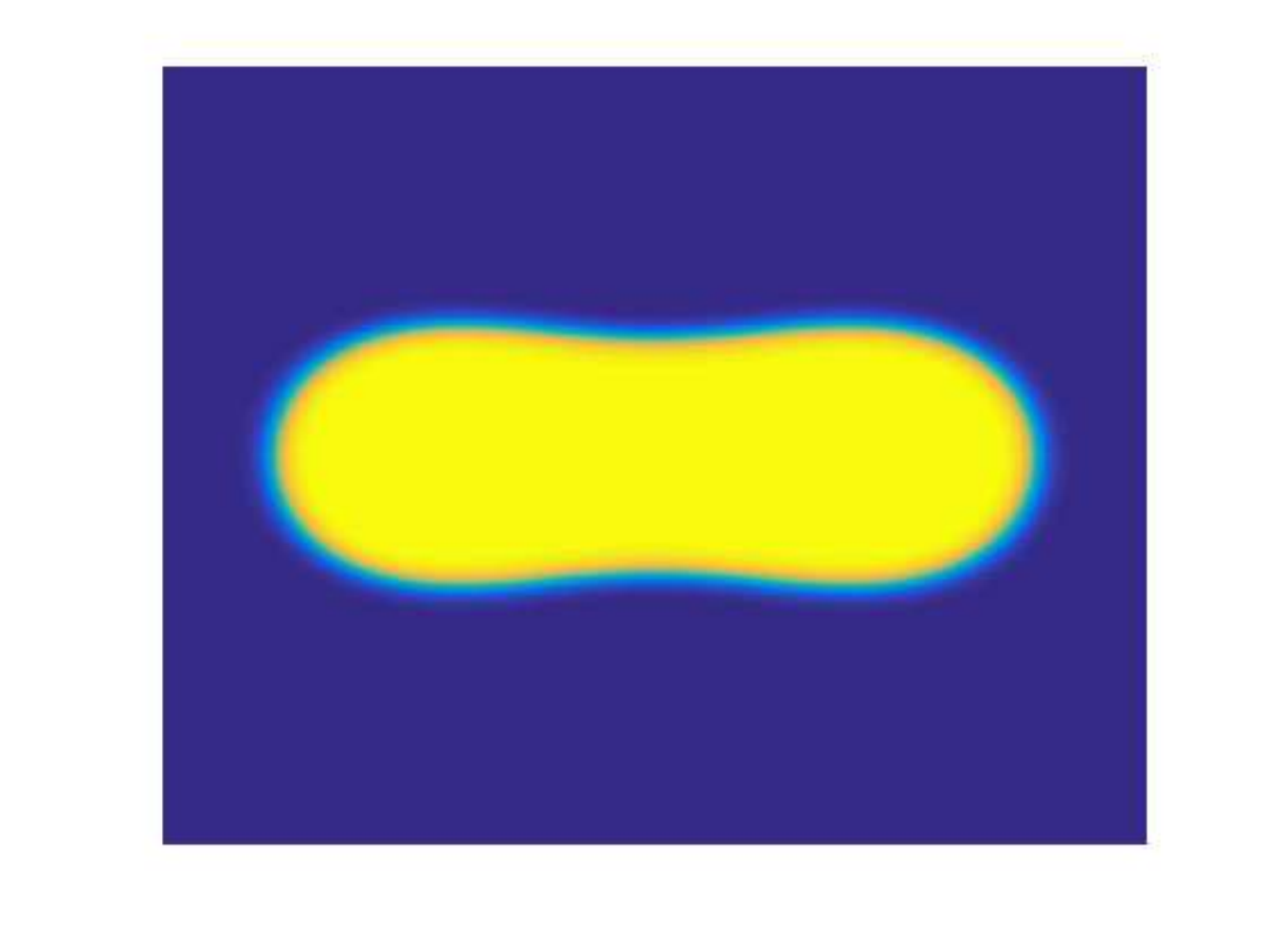}
}
\subfigure[t=0.4]
{
\includegraphics[width=3.5cm,height=3.5cm]{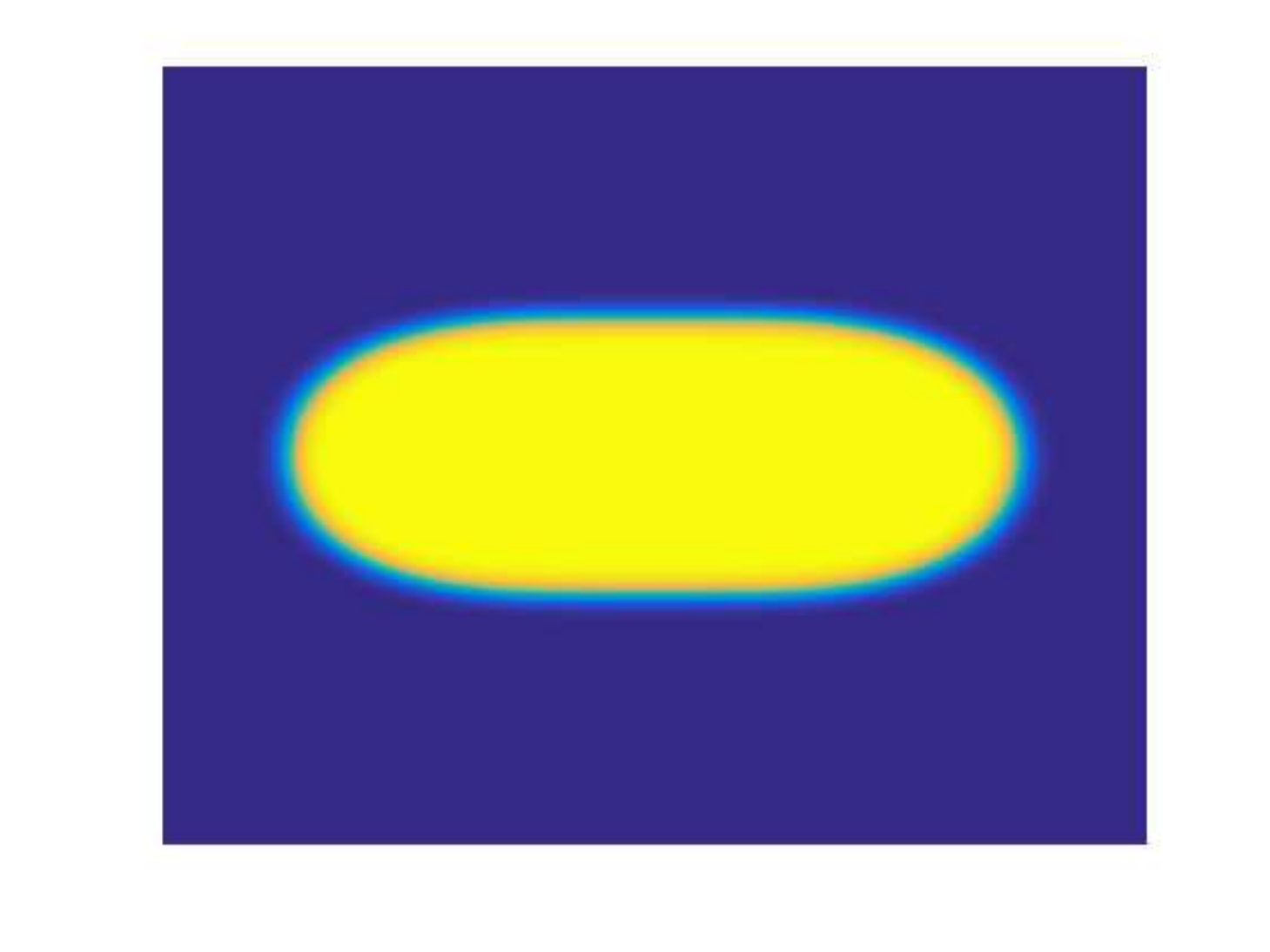}
}
\subfigure[t=0.8]
{
\includegraphics[width=3.5cm,height=3.5cm]{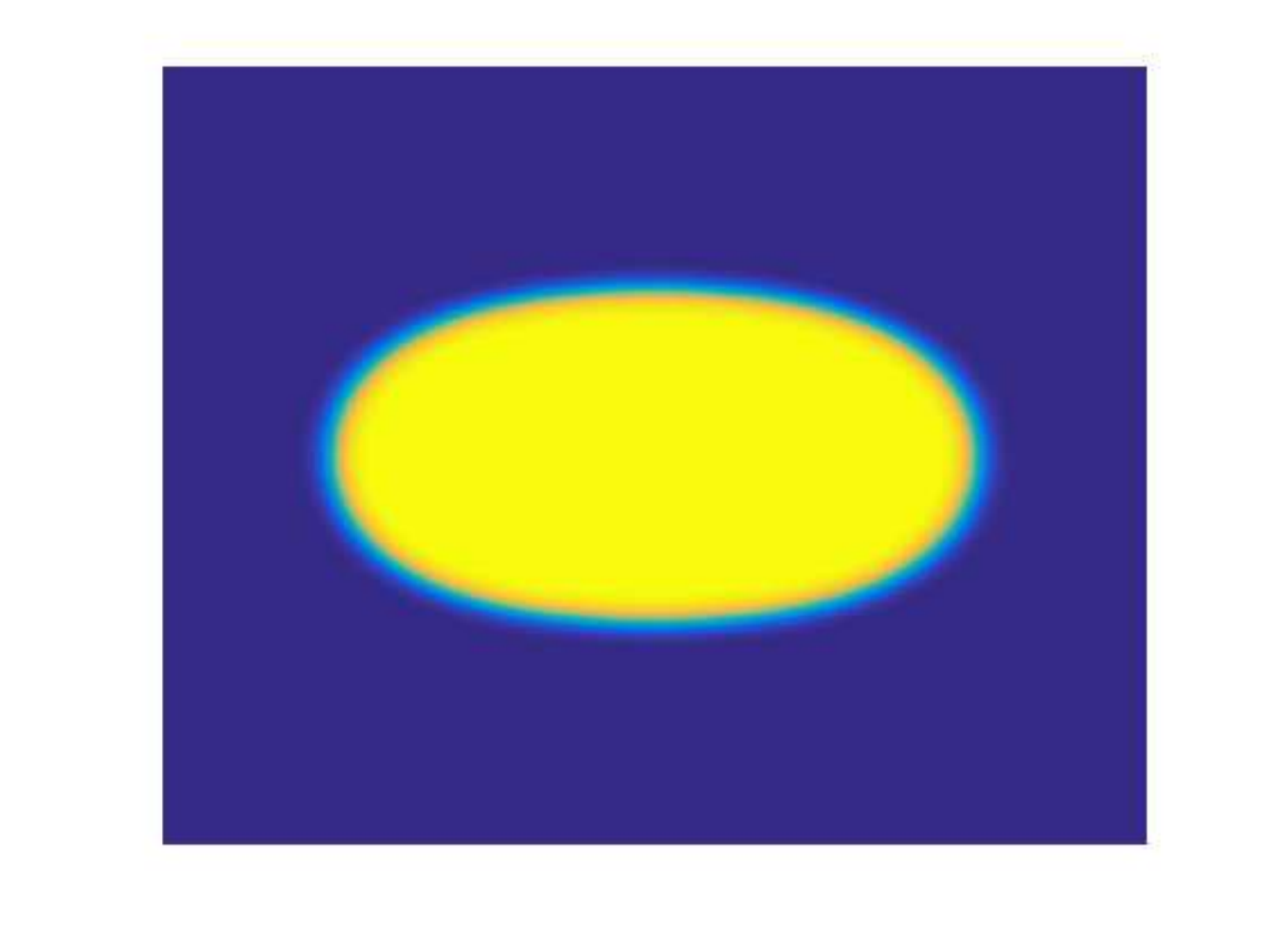}
}
\subfigure[t=1]
{
\includegraphics[width=3.5cm,height=3.5cm]{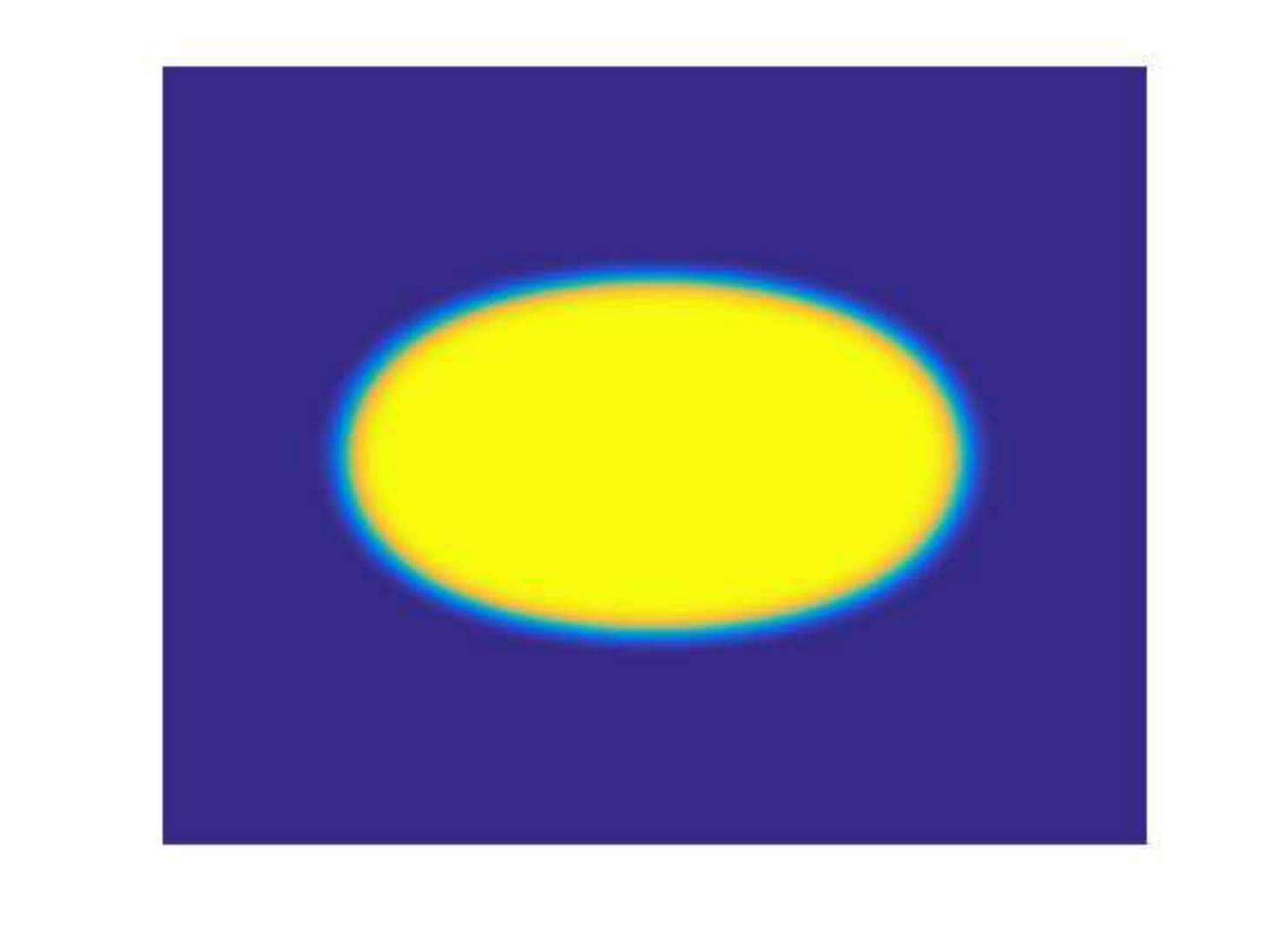}
}
\caption{Snapshots of the phase variable $\phi$ for Cahn-Hilliard Model are taken at t=0, $8e-4$, $4e-3$, $0.02$, $0.16$, $0.4$, $0.8$ and $1$ for example 2.}\label{ch-fig1}
\end{figure}

As is known to all, the Allen-Cahn equation does not conserve mass and the Cahn-Hilliard model preserves mass. So, in Figure \ref{ac-fig1}, we can see that the bubble shrinks and finally disappears. From Figure \ref{ch-fig1}, we see that as time evolves, the two bubbles coalesce into a single bubble. Finally, the shape of the bubble reaches a steady circular shape. A correct simulation of this phenomenon shows the effectiveness of our NAEV method.

\subsection{The phase field crystal model}
In the following example,  We first test convergence rates of the proposed schemes for the phase field crystal equation in two dimension and check the efficiency and accuracy. Then, we study the phase transition behavior of the density field. The phase field crystal model can be written as follows:
\begin{equation*}
\frac{\partial \phi}{\partial t}=\Delta\left(\phi^3-\epsilon\phi+(1+\Delta)^2\phi\right), \quad(\textbf{x},t)\in\Omega\times Q,
\end{equation*}
The relative Swift-Hohenberg free energy takes the form:
\begin{equation*}
E(\phi)=\int_{\Omega}\left(\frac{1}{4}\phi^4+\frac{1}{2}\phi\left(-\epsilon+(1+\Delta)^2\right)\phi\right)d\textbf{x},
\end{equation*}
where $\textbf{x} \in \Omega \subseteq \mathbb{R}^d$, $\phi$ is the density field and $\epsilon$ is a positive bifurcation constant with physical significance. $\Delta$ is the Laplacian operator.

\textbf{Example 3}: Consider the following initial condition
\begin{equation*}
\aligned
\phi(x,y,0)=\sin(\frac{2\pi x}{64})\sin(\frac{2\pi y}{64}).
\endaligned
\end{equation*}

The computational domain is set to be $\Omega=[0,128]\times[0,128]$ and the order parameters are $\epsilon=0.25$, $T=1$.

We list the $L^2$ errors and temporal convergence rates of the phase variable for both first-order and second-order NAEV schemes. We find that our NAEV schemes can all achieve almost perfect first and second order accuracy in time.
\begin{table}[h!b!p!]
\small
\renewcommand{\arraystretch}{1.1}
\centering
\caption{\small The $L_2$ errors, convergence rates of phase field crystal equation for the first-order and second-order NAEV schemes in time at $T=1$.}\label{tab:tab3}
\begin{tabular}{cccccccccccc}
\hline
        &      &        &1st-Order NAEV&        &&       &2nd-Order NAEV&\\
\cline{2-5}\cline{7-10}
$\Delta t$   &$1/2^4$  &$1/2^5$    &$1/2^6$   &$1/2^7$   &&$1/2^4$  &$1/2^5$  &$1/2^6$    &$1/2^7$\\
\hline
Err          &7.65e-3 &3.82e-3   &1.91e-3  &9.52e-4 &&3.45e-4 &8.69e-5 &2.18e-5   &5.46e-6\\
Rate         &-       &1.00      &1.00     &1.00     &&-       &1.99    &2.00      &2.00\\
\hline
\end{tabular}
\end{table}

\textbf{Example 4}: In the following, we take the order parameter is $\epsilon=0.025$, $\Omega=[-50,50]^2$, $\tau=1/10$. The initial condition is
\begin{equation}\label{section5_e2}
\aligned
&\phi_0(x,y)=0.07+0.07\times rand(x,y),
\endaligned
\end{equation}
where the $rand(x,y)$ is the random number in $[-1,1]$ with zero mean.

In this example, we show the phase transition behavior of the density field for different values at various times in Figure \ref{fig1}. Similar computation results for phase field crystal model can be found in \cite{yang2017linearly} which verifies the efficiency of our NAEV approach.
\begin{figure}[htp]
\centering
\subfigure[t=4]{
\includegraphics[width=4cm,height=4cm]{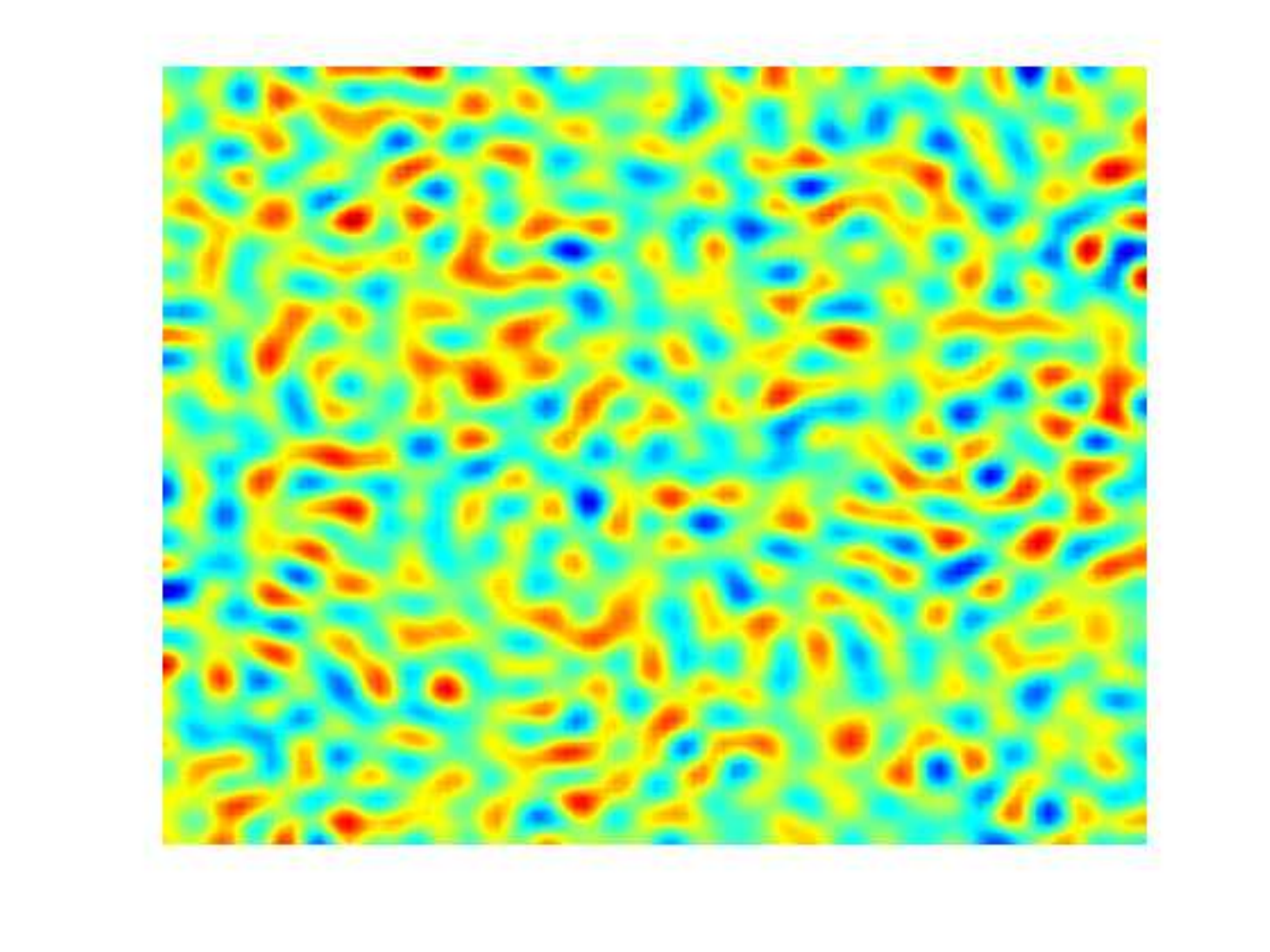}
}
\subfigure[t=15]
{
\includegraphics[width=4cm,height=4cm]{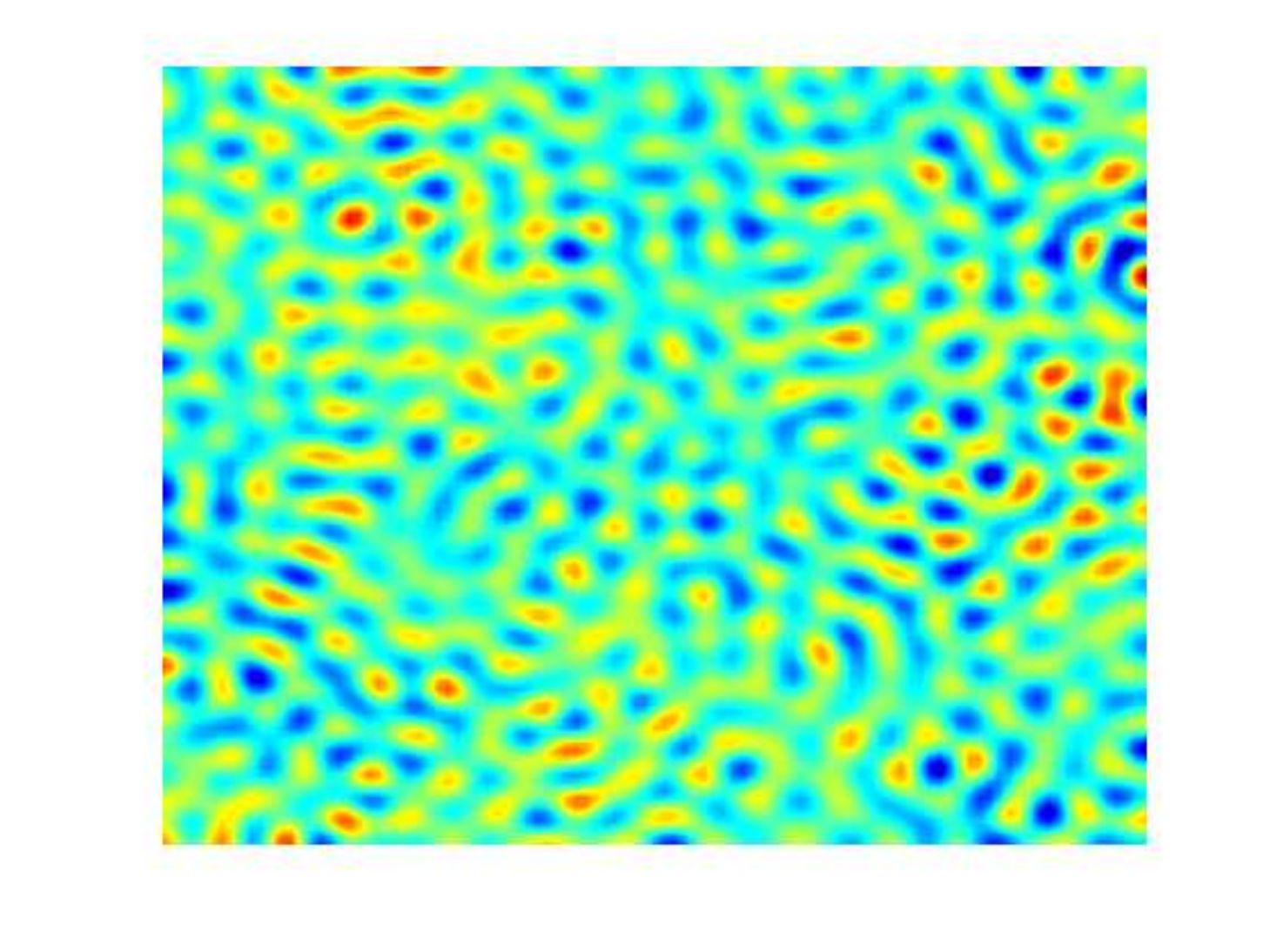}
}
\subfigure[t=100]
{
\includegraphics[width=4cm,height=4cm]{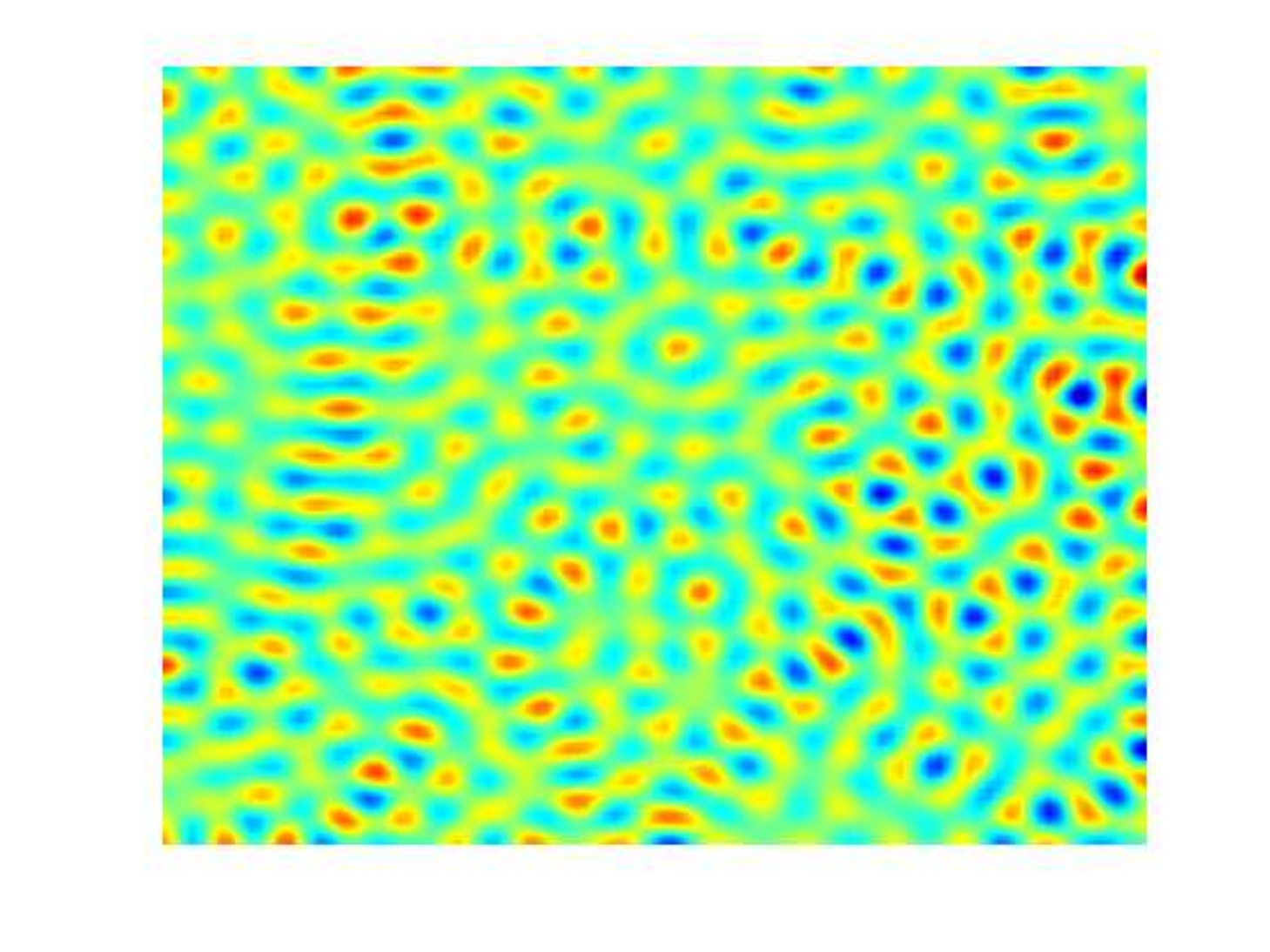}
}
\quad
\subfigure[t=200]{
\includegraphics[width=4cm,height=4cm]{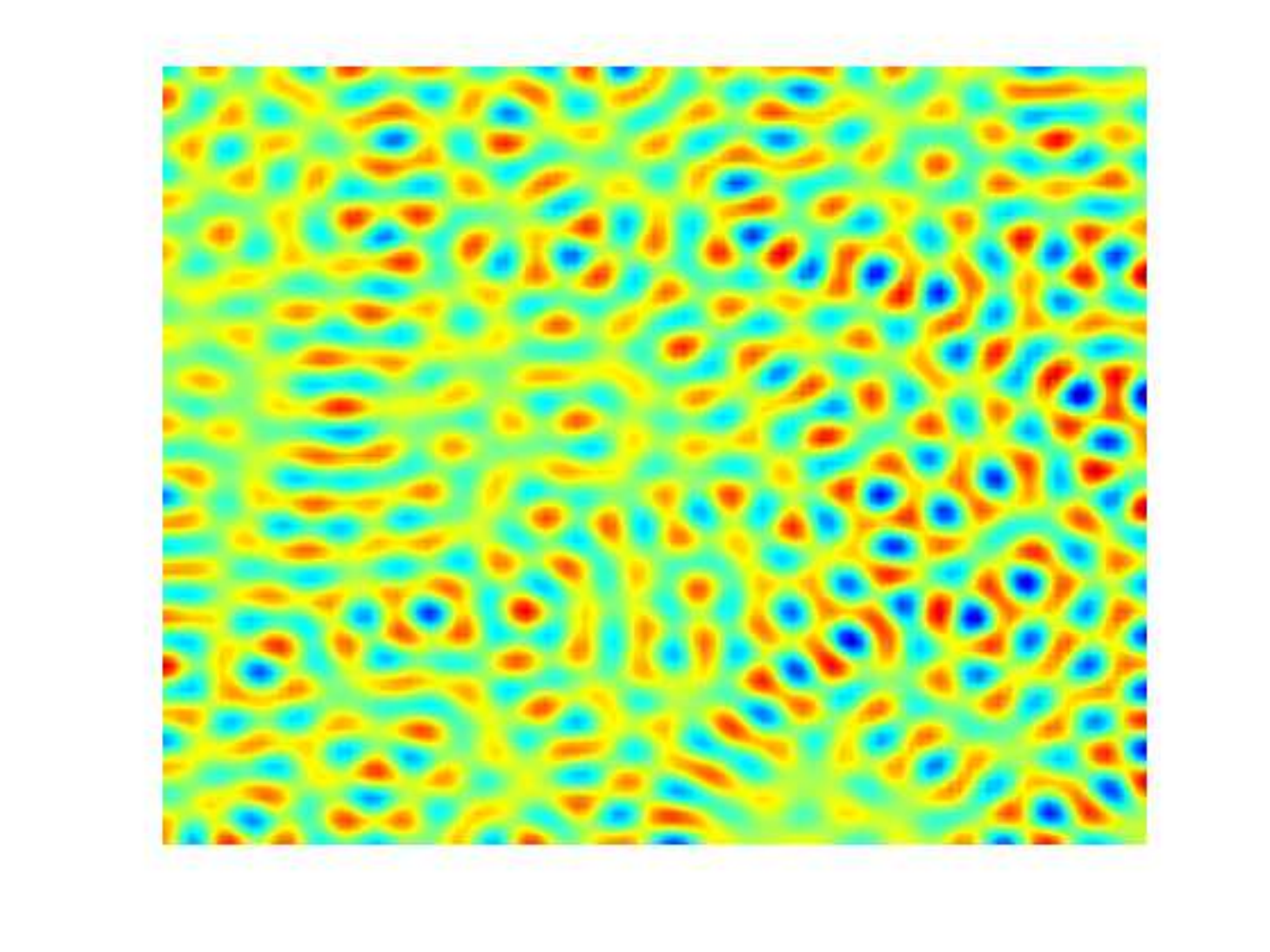}
}
\subfigure[t=400]
{
\includegraphics[width=4cm,height=4cm]{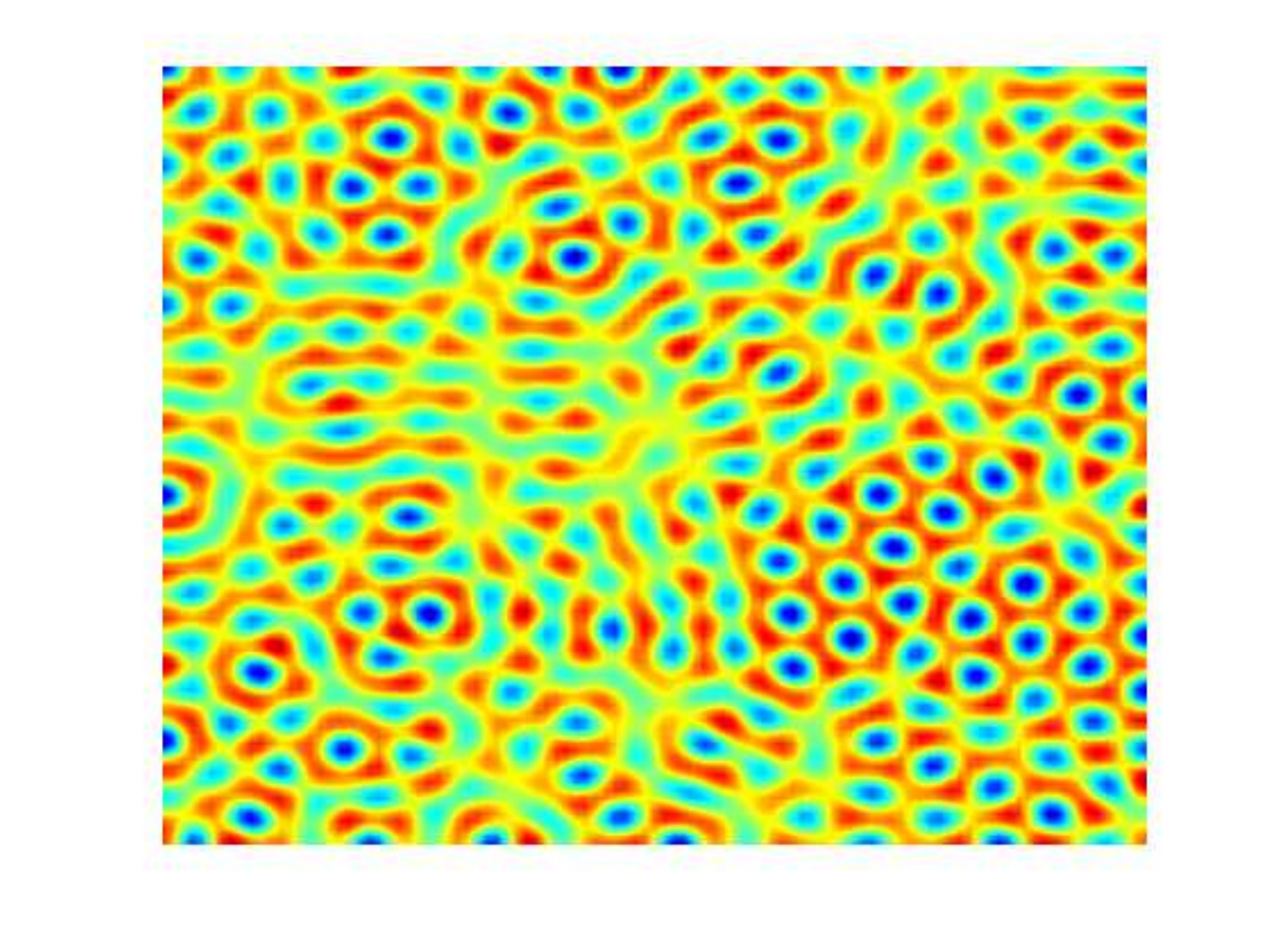}
}
\subfigure[t=1000]
{
\includegraphics[width=4cm,height=4cm]{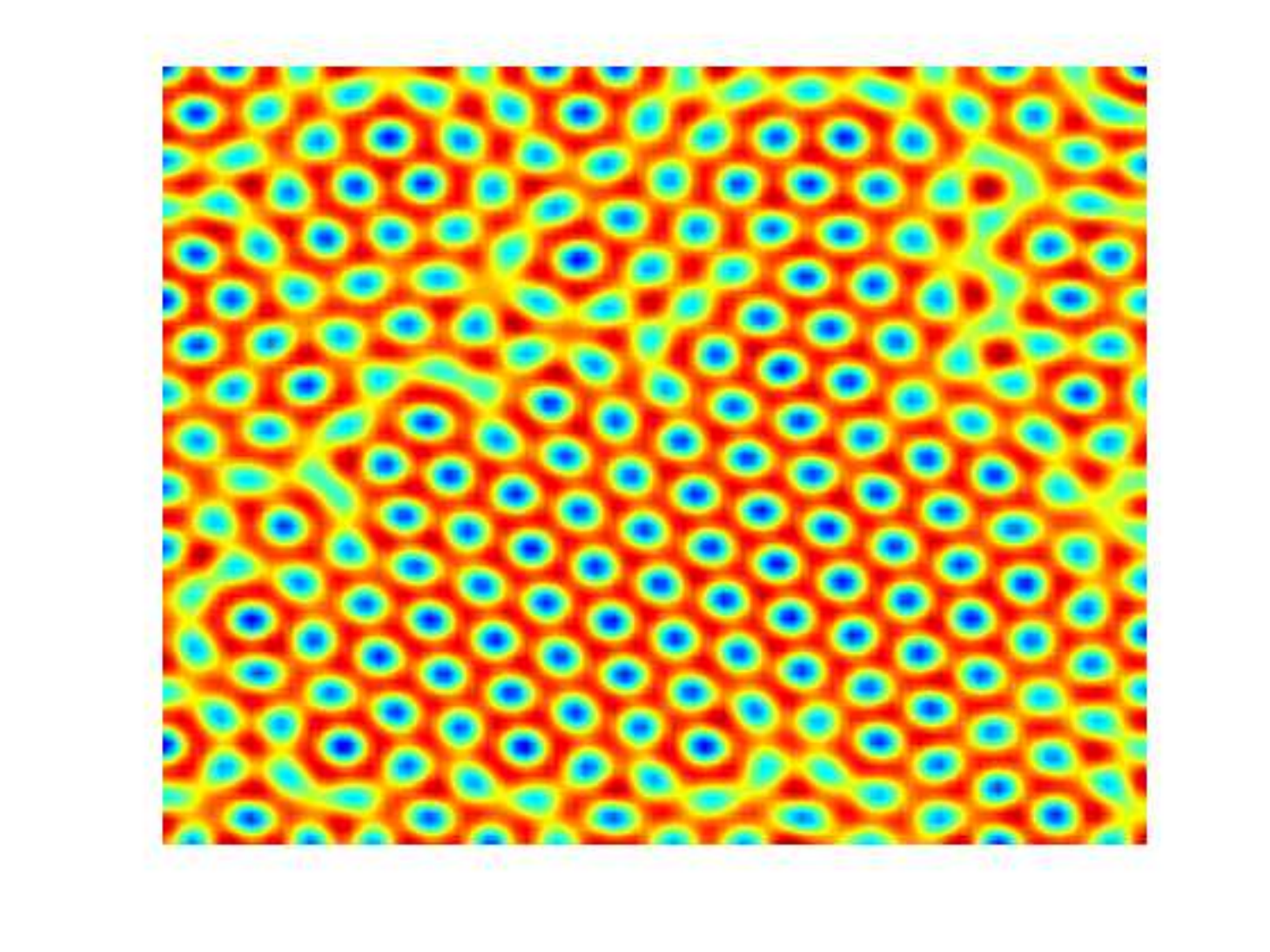}
}
\caption{Snapshots of the phase variable $\phi$ are taken at t=4, 15, 100, 200, 400, 1000 for example 4.}\label{fig1}
\end{figure}

\subsection{The Swift-Hohenberg model with quadratic-cubic nonlinearity}
In this subsection, we aim to study a comparative study of classical SAV and NAEV approaches. In the following example, we study the Swift-Hohenberg (SH) equation with quadratic-cubic nonlinearity to check the efficiency of NAEV approach. Given the following free energy functional \cite{lee2019energy}:
\begin{equation*}
\aligned
&E(\phi)=\int_\Omega\frac14\phi^4-\frac g3\phi^3+\frac12\phi\left(-\epsilon+(1+\Delta)^2\right)\phi d\textbf{x}.
\endaligned
\end{equation*}
where $\phi$ is the density field and $g\geq0$ and $\epsilon>0$ are constants with physical significance.

It is obvious that $E_1(\phi)=\int_\Omega\frac14\phi^4-\frac g3\phi^3-\frac12\epsilon\phi^2d\textbf{x}$ will be negative in some cases because of $-\int_\Omega\frac g3\phi^3+\frac12\epsilon\phi^2d\textbf{x}$ for $g>0$.

The SH model can be modeled by $L^2$-gradient flow from the energetic variation of the above energy functional $E(\phi)$:
\begin{equation}\label{section5_sh2}
\aligned
&\frac{\partial \phi}{\partial t}=-(\phi^3-g\phi^2+(-\epsilon+(1+\Delta)^2\phi)),
\endaligned
\end{equation}

Consider the following example:

\textbf{Example 5}: The initial condition which can be seen in \cite{lee2019energy} is
\begin{equation*}
\aligned
\phi_0(x,y)=
&0.07-0.02\cos\left(\frac{2\pi(x-12)}{32}\right)\sin\left(\frac{2\pi(y-1)}{32}\right)+0.02\cos^2\left(\frac{\pi(x+10)}{32}\right)\sin^2\left(\frac{\pi(y+3)}{32}\right)\\
&-0.01\sin^2\left(\frac{4\pi x}{32}\right)\sin^2\left(\frac{4\pi(y-6)}{32}\right),
\endaligned
\end{equation*}
on $\Omega=[0,128]^2$. We use $g=2$, $\epsilon=0.25$, $h=0.01$.

We list the $L^2$ errors and temporal convergence rates of the phase variable at $T=1$ with different time step sizes by choosing constant parameter in square root $C=10$ and $C=100$ for SAV scheme and $\kappa=0$ for NAEV scheme. We find that the SAV scheme with $C=100$ and NAEV scheme can all achieve almost perfect second order accuracy in time. However, the SAV approach with $C=10$ is failed to obtained the correct errors and temporal convergence rates. The reason is that the value of $E_1(\phi)$ becomes less than $-10$ as time goes on. It means that we may not give a proper constant $C$ before calculation in SAV approach. But for NAEV approach, we do not need to give a constant $C$ before and during calculation any more. Such advantages will contribute to the applicability of SAV approach.

\begin{table}[h!b!p!]
\small
\renewcommand{\arraystretch}{1.1}
\centering
\caption{\small The $L_2$ errors, temporal convergence rates of SAV and NAEV schemes for Example 5.}\label{tab:tab5}
\begin{tabular}{ccccccccccc}
\hline
$\Delta t$&\multicolumn{2}{c}{SAV-CN,C=10}&\multicolumn{3}{c}{SAV-CN,C=100}&\multicolumn{2}{c}{NAEV-CN}\\
\cline{2-6}\cline{7-8}
&$L_2$ error&Rate&&$L_2$ error&Rate&$L_2$ error&Rate\\
\hline
$2^{-4}$  &2.5570e-2&-    &&1.4842e-3&-   &1.1039e-3&-     \\
$2^{-5}$  &1.1924e-2&1.10 &&3.8859e-4&1.93&2.8735e-4&1.94      \\
$2^{-6}$  &5.1724e-3&1.21 &&9.9366e-5&1.97&7.3269e-5&1.97     \\
$2^{-7}$  &9.0814e-4&2.51 &&2.5120e-5&1.98&1.8497e-5&1.98       \\
$2^{-8}$  &1.1759e-2&-3.6 &&6.3150e-6&1.99&4.6468e-6&1.99      \\
$2^{-9}$  &4.9277e-3&1.25 &&1.5831e-6&2.00&1.1645e-6&2.00      \\
$2^{-10}$ &2.6940e-4&4.19 &&3.9633e-7&2.00&2.9148e-7&2.00     \\
\hline
\end{tabular}
\end{table}

\textbf{Example 6}: set $\epsilon=0.25$, $g=2$. The initial condition is
\begin{equation}\label{section5_sh3}
\aligned
&\phi_0(x,y)=0.1+0.1rand(x,y),
\endaligned
\end{equation}

Noting that the energy of SH model will reduce to be negative energy. The functional $E_1(\phi)$ in square root will be farther and farther away from zeros. It means that a very big constant $C$ is needed to keep the square root reasonable. In practice, from the left one in Figure \ref{fig2}, $E_1(\phi)$ will drop to about $-6100$. If we set $C<6100$ for SAV approach, $E_1(\phi)+C\geq0$ will not satisfied for some $\phi$. Thus, we choose $C=10000$ in SAV scheme. However, in NAEV approach, we only need to compute the initial energy $E_0$ before calculation and set the parameter $\kappa=0$. In Figure \ref{fig2}, we also show the energy evolution for the SH model when using SAV and NAEV approaches. One can see that the NAEV approach is more efficient than SAV approach. Figure \ref{fig3} shows the evolution of $\phi(x,y,t)$ using NAEV scheme with $\Delta t=1/20$. The similar features to those of SH model can obtain in \cite{lee2019energy}.
\begin{figure}[htp]
\centering
\includegraphics[width=7cm,height=7cm]{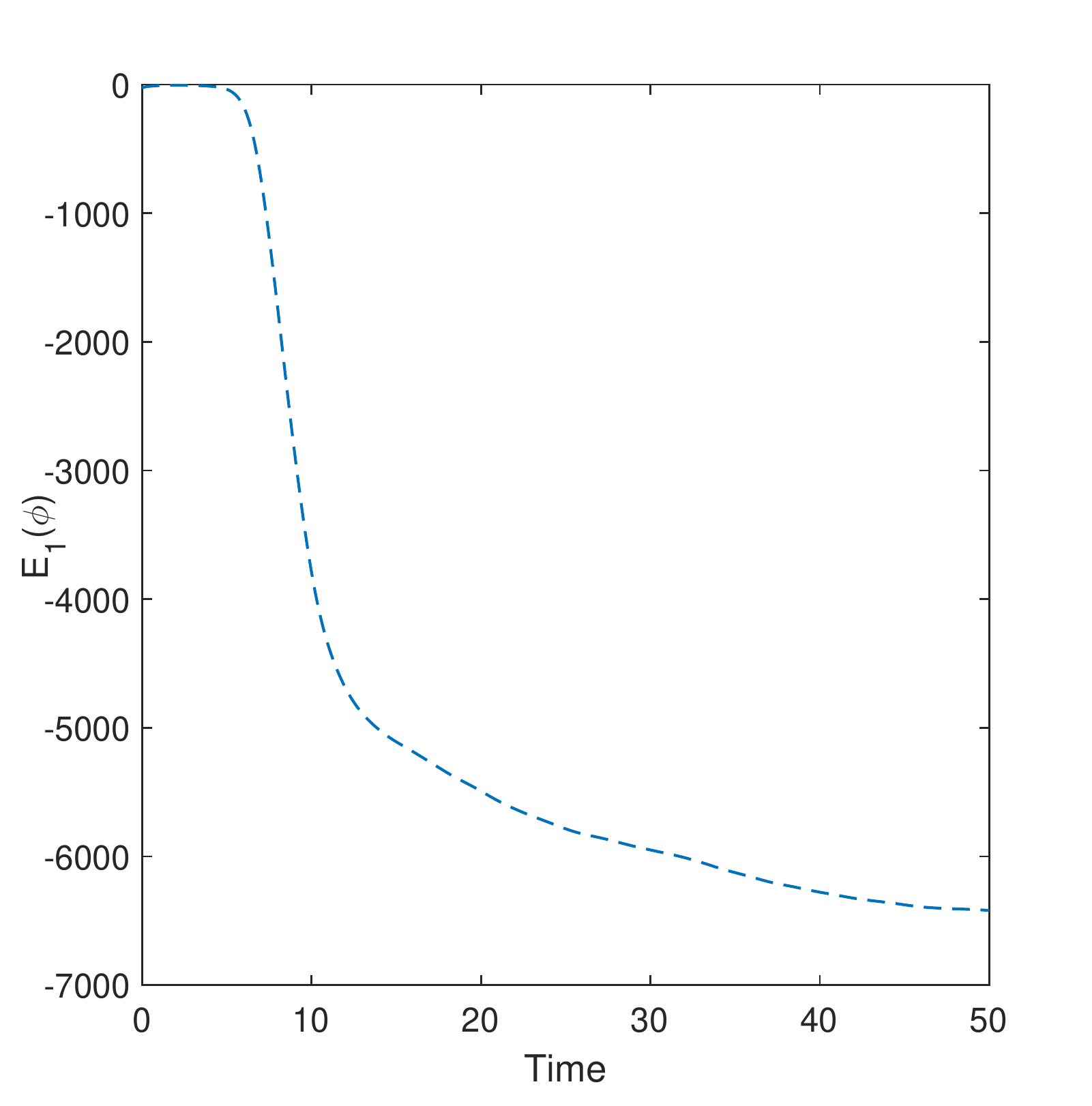}
\includegraphics[width=7cm,height=7cm]{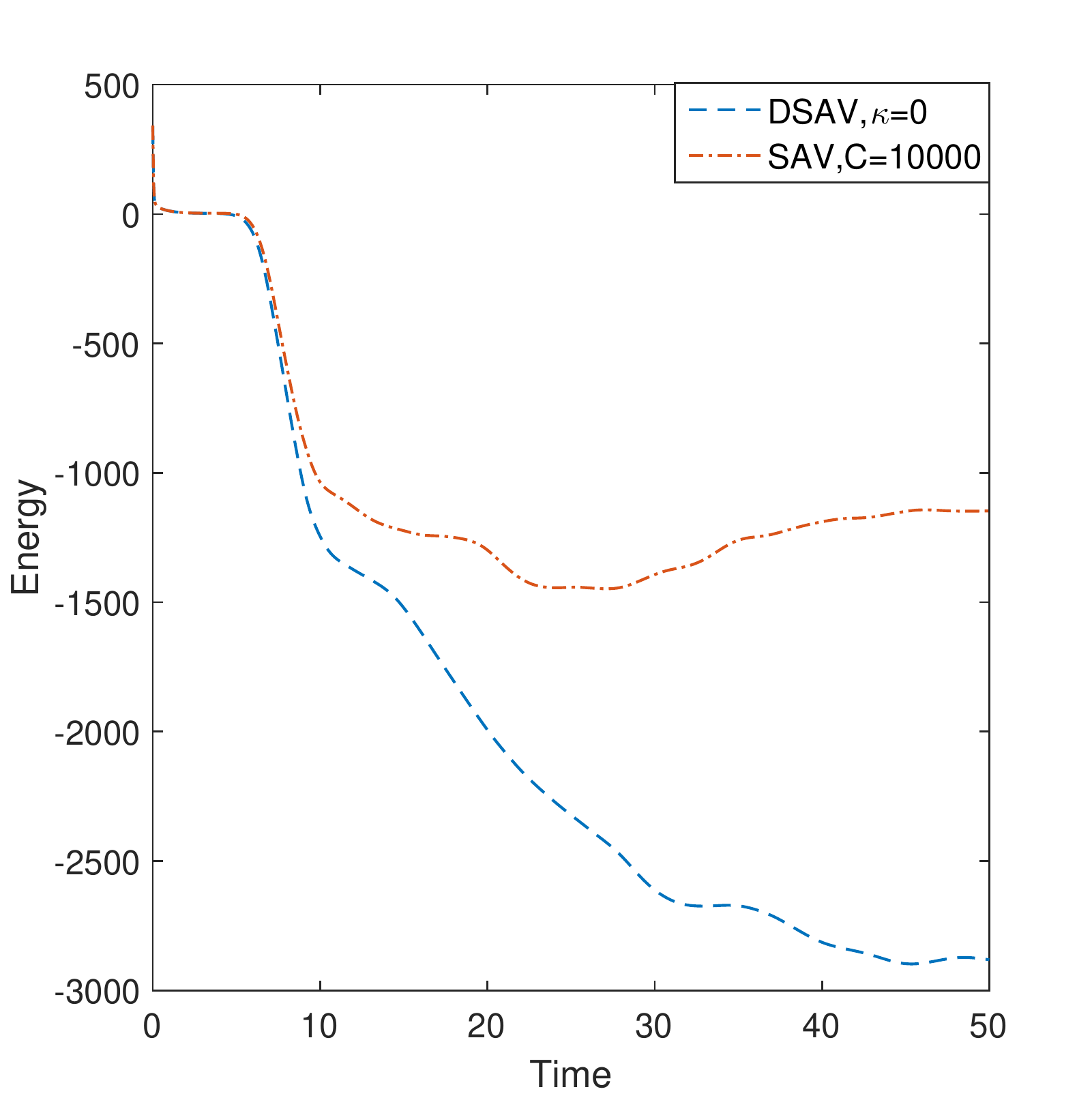}
\caption{Left: the value of $E_1(\phi)$ for example 6 with $\Delta t$=1/20. Right: energy evolution for SH equation for example 6 using SAV and NAEV approaches.}\label{fig2}
\end{figure}
\begin{figure}[htp]
\centering
\subfigure[t=2]{
\includegraphics[width=4cm,height=4cm]{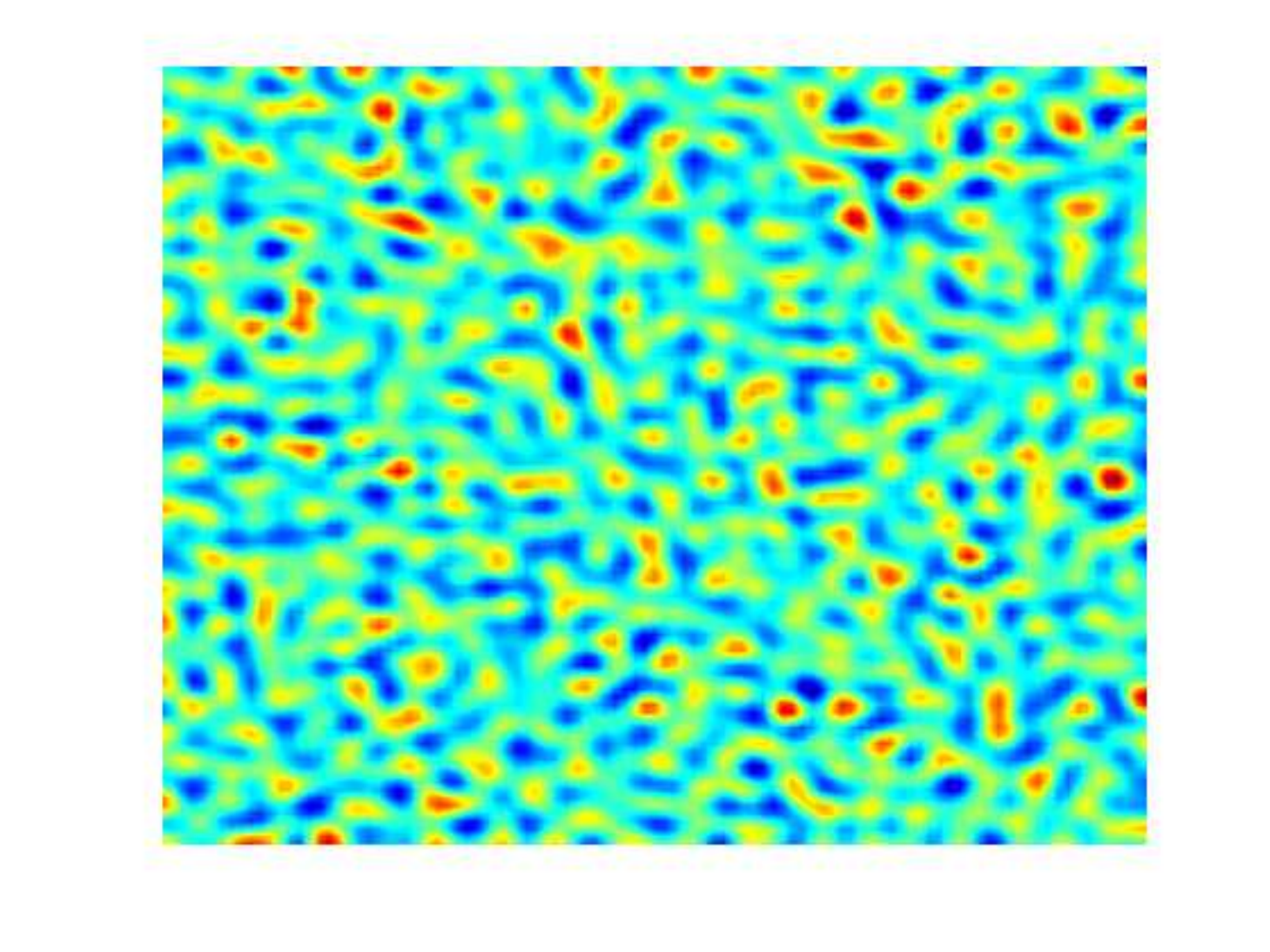}
}
\subfigure[t=10]
{
\includegraphics[width=4cm,height=4cm]{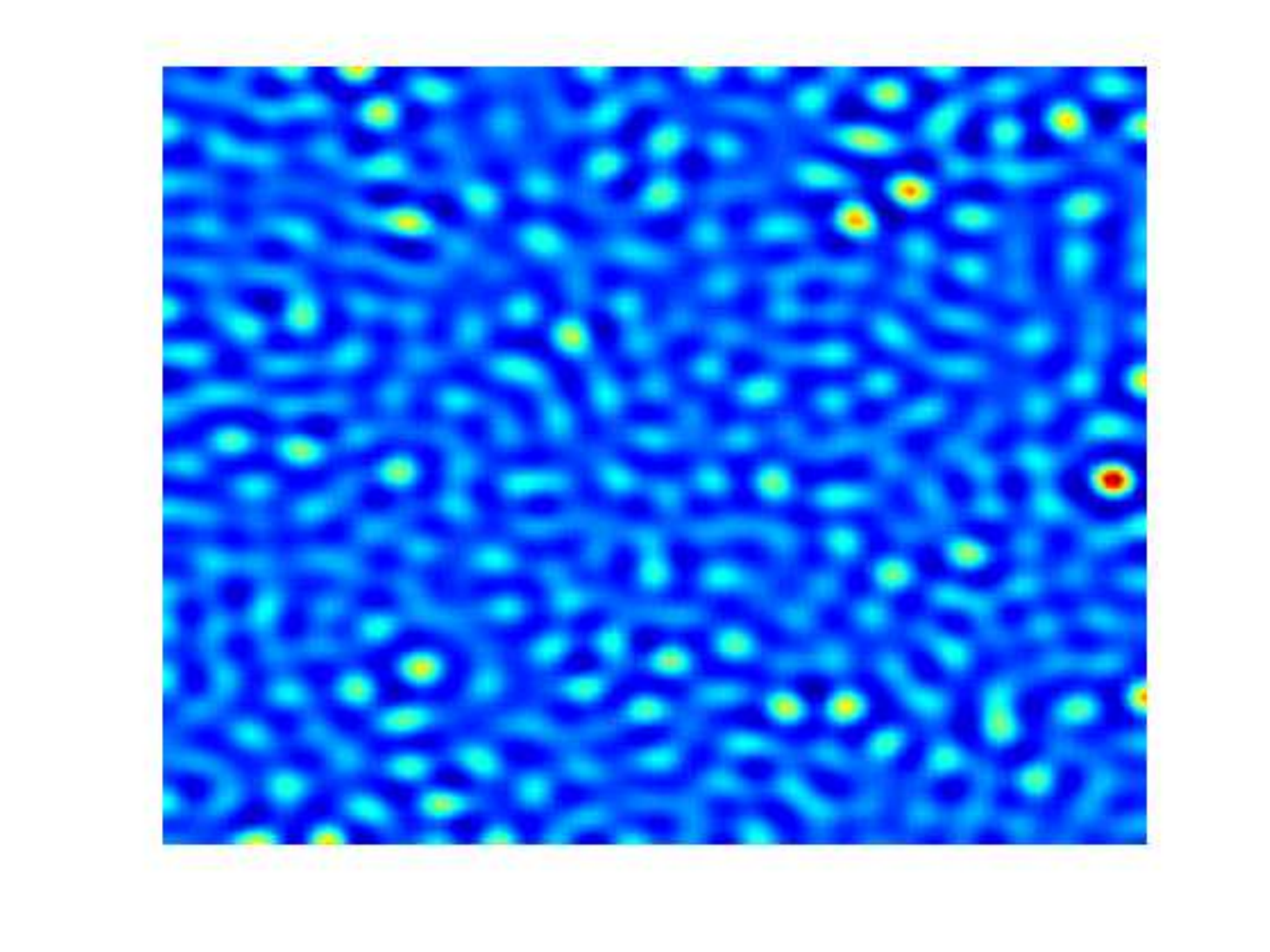}
}
\subfigure[t=30]
{
\includegraphics[width=4cm,height=4cm]{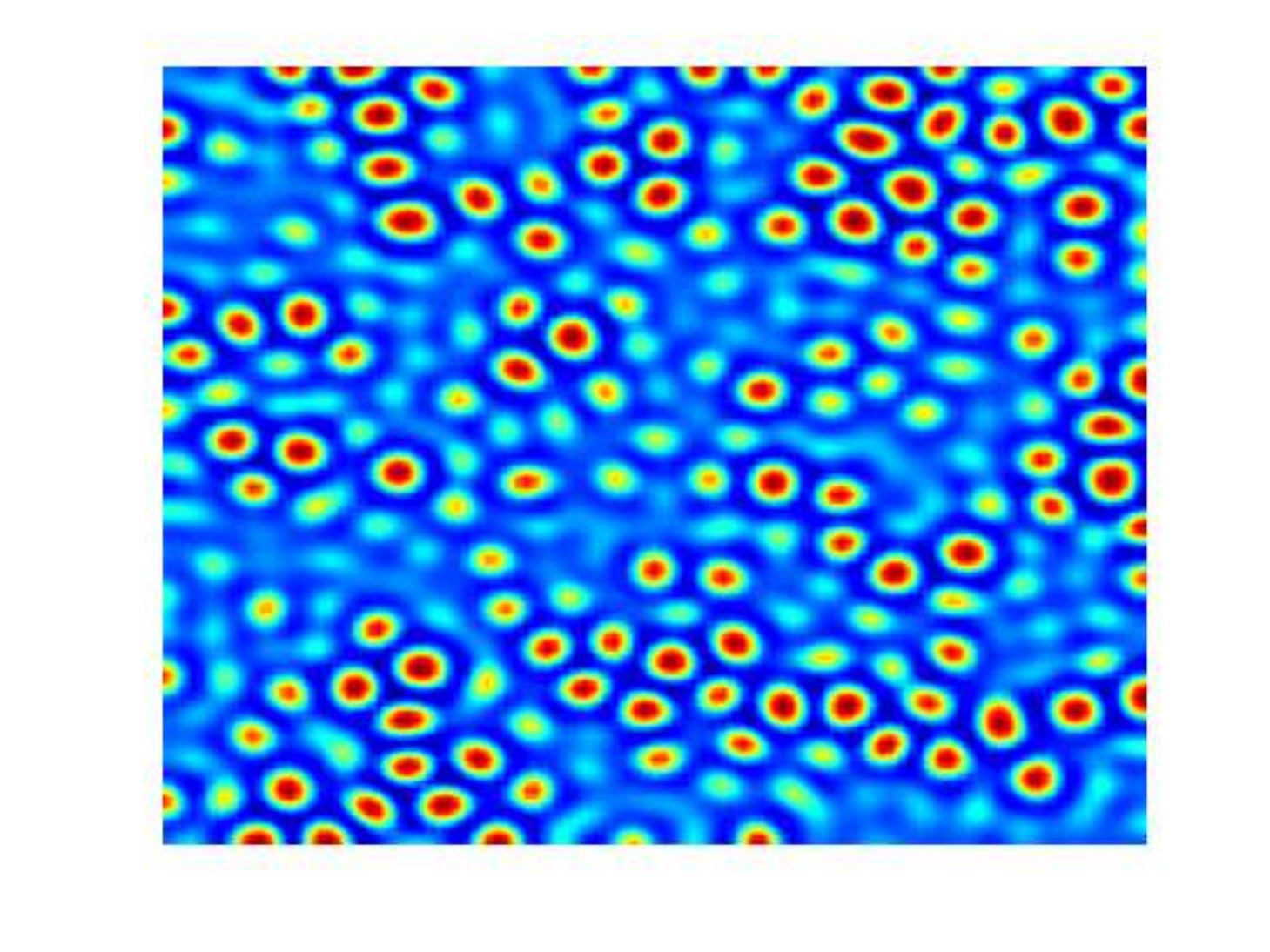}
}
\quad
\subfigure[t=40]{
\includegraphics[width=4cm,height=4cm]{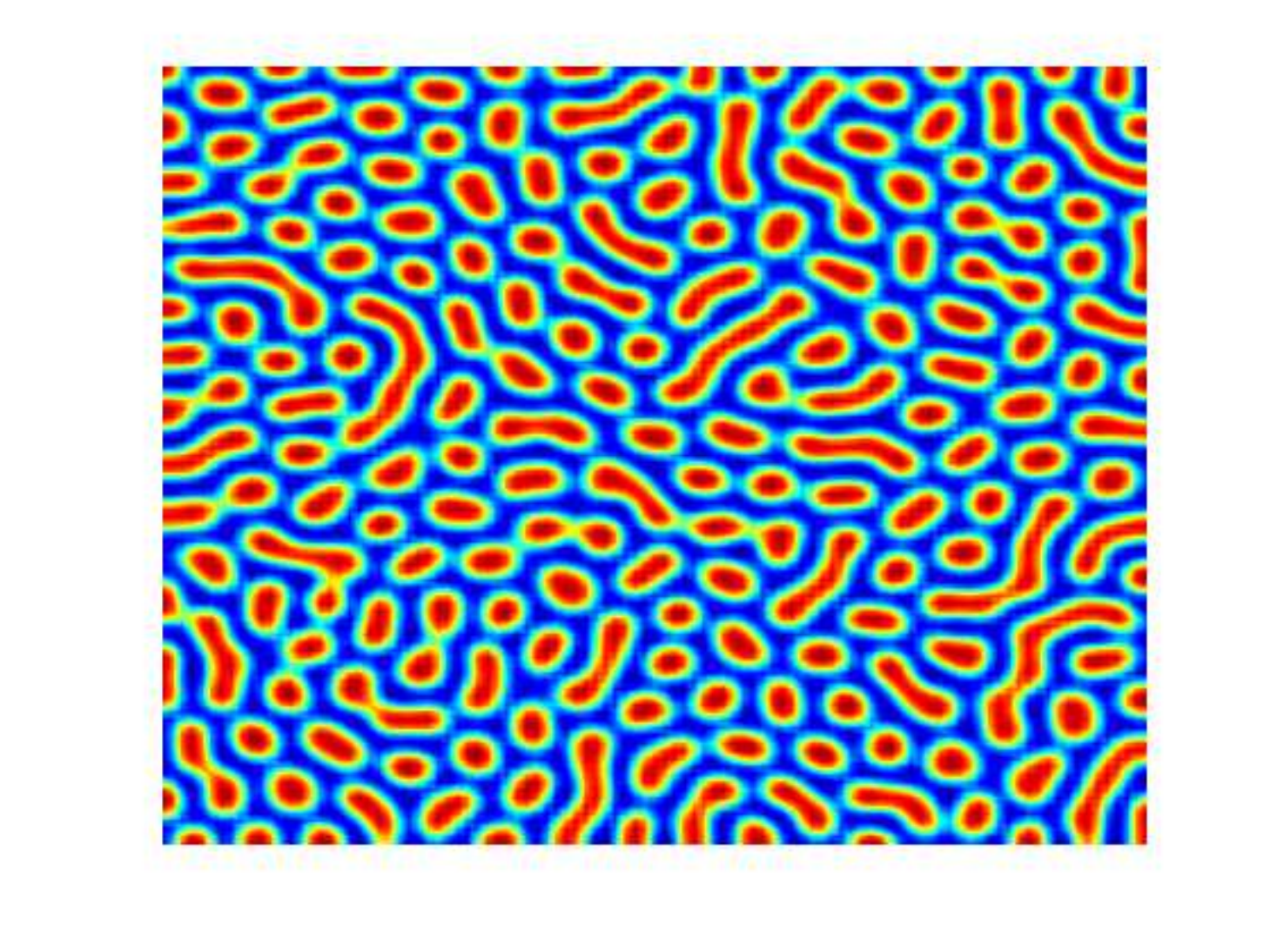}
}
\subfigure[t=60]
{
\includegraphics[width=4cm,height=4cm]{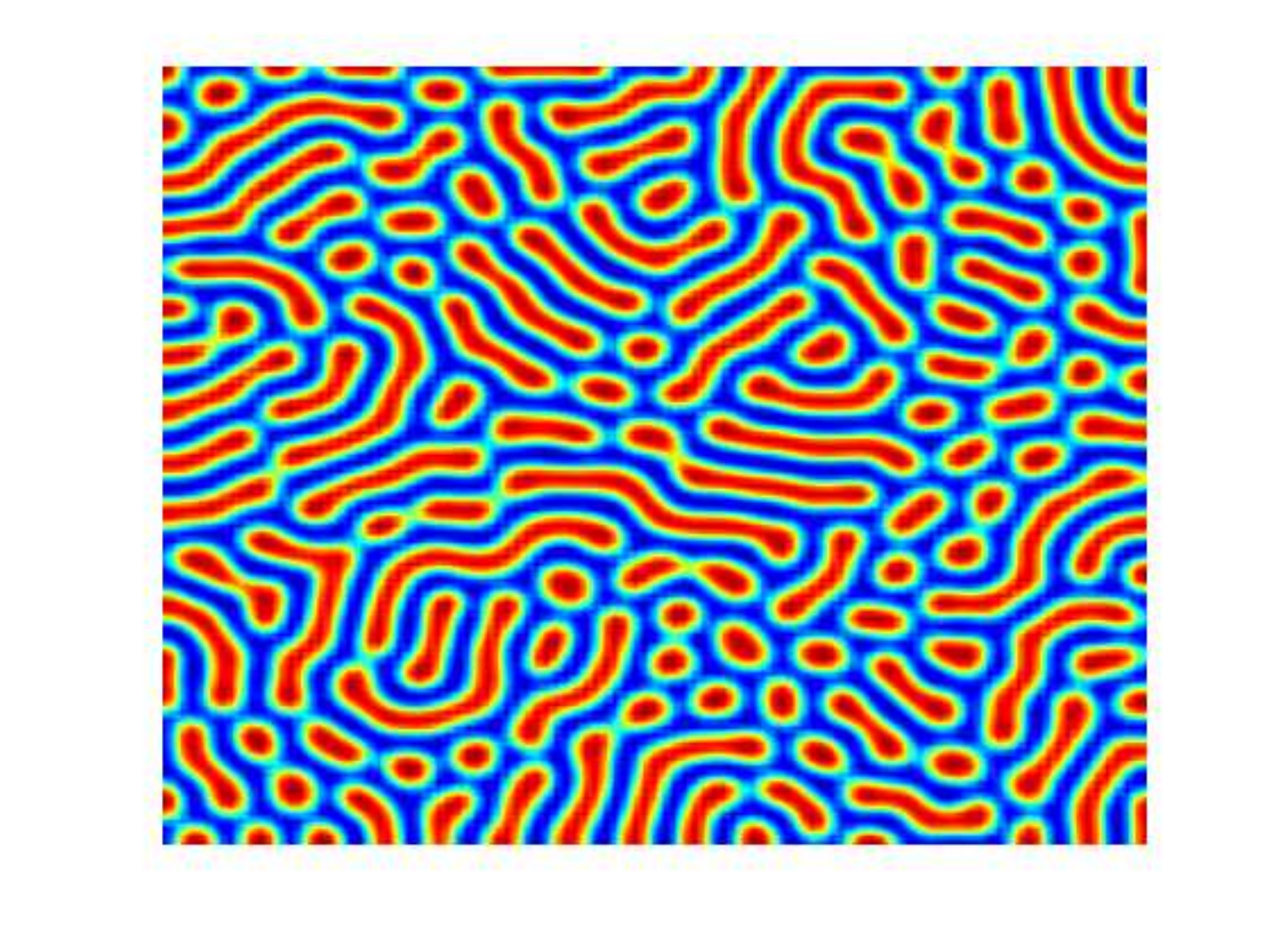}
}
\subfigure[t=120]
{
\includegraphics[width=4cm,height=4cm]{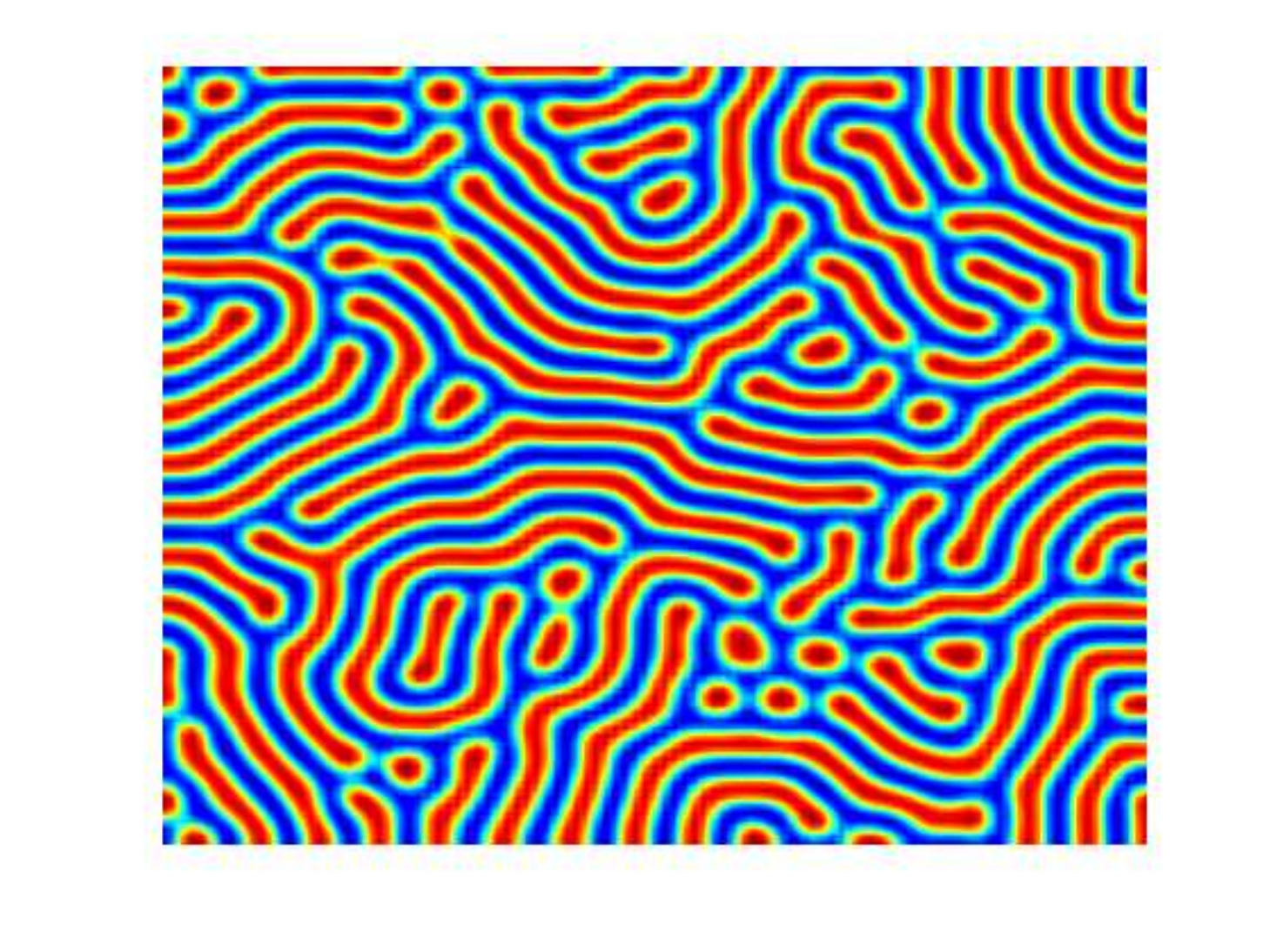}
}
\caption{Snapshots of the phase variable $\phi$ are taken at t=2, 10, 30, 40, 60, 120 for example 6.}\label{fig3}
\end{figure}
\section{Conclusion}
In this paper, we consider a novel auxiliary variable method to obtain energy stable schemes for gradient flows and prove the unconditional energy stability for the given semi-discrete schemes carefully and rigorously. Both first-order and second-order NAEV schemes are considered and proved to be efficiently.  A comparative study of SAV and NAEV approaches is considered to show the accuracy and efficiency. Finally, we present various 2D numerical simulations to demonstrate the stability and accuracy.
\section*{Acknowledgement}
No potential conflict of interest was reported by the author. We would like to acknowledge the assistance of volunteers in putting together this example manuscript and supplement.
\bibliographystyle{siamplain}
\bibliography{Deformed-SAV}

\end{document}